\documentclass{amsart}
\usepackage{graphicx} 
\usepackage{amsmath, amssymb, amsthm}
\usepackage[mathcal]{eucal}
\usepackage{enumerate}
\usepackage{amsfonts}
\usepackage{amscd}
\usepackage{mathrsfs}
\usepackage{upgreek}
\usepackage{dirtytalk}
\usepackage{verbatim}
\usepackage{MnSymbol}
\usepackage{enumitem}
\usepackage{tikz}

\definecolor{midnightblue}{HTML}{0059b3}
\definecolor{chromered}{HTML}{f14233}
 \RequirePackage[colorlinks,citecolor=midnightblue,urlcolor=midnightblue,breaklinks=true,
linkbordercolor=midnightblue,linkcolor=midnightblue]{hyperref}

\usepackage{cleveref}

\newtheorem{theorem}{Theorem}

\newtheorem{lemma}{Lemma}

\numberwithin{equation}{section}

\def\ti{\tilde}

\def\ZZ{\ensuremath{\mathbb Z}}
\def\C{\ensuremath{\mathbb C}}
\def\D{\ensuremath{\mathbb D}}
\def\N{\ensuremath{\mathbb N}}
\def\T{\ensuremath{\mathbb T}}
\def\R{\ensuremath{\mathbb R}}
\def\Z1{\ensuremath{\mathbf 1}}

\def\fmat{{\mathcal{F}}}
\def\zmat{{{Z}}}
\def\gmat{{{G}}}
\def\phimat{{\mathcal{P}}}
\def\hmat{{\mathcal{H}}}
\def\qmat{{\mathcal{Q}}}
\def\tr{{\rm tr}}

\renewcommand{\phi}{\varphi}
\renewcommand{\epsilon}{\varepsilon}
\renewcommand{\tilde}{\widetilde}
\renewcommand{\Re}{\operatorname{Re}}
\renewcommand{\Im}{\operatorname{Im}}

\newcommand{\ct}[1]{{\color{blue} CT: #1}}
\newcommand{\meas}[1]{\frac{d|#1|}{2 \pi}}

\title[On one-sided OPUC and pointwise convergence for the NLFS]{One sided orthogonal polynomials and a pointwise convergence result for $SU(2)$-valued nonlinear Fourier series}

\author{Michel Alexis}
\author{Gevorg Mnatsakanyan}
\author{Christoph Thiele}

\address{Mathematical Institute, 
	University of Bonn,
	Endenicher Allee 60, 53115 Bonn,
	Germany}
\email{alexis@math.uni-bonn.de}
\email{thiele@math.uni-bonn.de}
\address{Institute of Mathematics of the National Academy of Sciences of Armenia, Marshal Baghramyan Ave. 24/5, Yerevan 0019, Armenia}

\email{mnatsakanyan\_g@yahoo.com}

\begin{document}

\begin{abstract}
    We elaborate on a connection between the $SU(2)$-valued nonlinear Fourier series and sequences of left and right orthogonal polynomials for complex measures on the unit circle. We show a convergence result for the associated reproducing kernel. This is a universality type result in the vein of Mate-Nevai-Totik, which turns out to be much simpler in the $SU(2)$ case than in the $SU(1,1)$ case. We then relate a.e.\ pointwise convergence of the product of left and right polynomials and their squares with both behavior of their zeros as well as behavior of some local parameters for these polynomials. We conclude by proving almost everywhere convergence along lacunary sequences of the functional $(a_n ^* +b_n)(a_n - b_n ^*)$ of  the partial $SU(2)$-valued nonlinear Fourier series $(a_n, b_n)$ under the assumption that the nonlinear Fourier series $(a,b)$ itself satisfies both $\|b\|_{L^{\infty} (\T)} < 2^{- \frac 1 2}$ and $a^*$ is outer. 
\end{abstract}

\maketitle

\tableofcontents

\section{Introduction}

For a complex measure $\mu$ on the unit circle $\T \subset \C$ and polynomials $f,g$ in a complex variable, define
\begin{equation}\label{defquasihilbert}
    \langle f , g\rangle_{\mu} := \int\limits_{\T} f g^* d \mu \, , \end{equation}
where, for functions $g$ in one or several complex variables,  we define
\begin{equation}\label{stardef}
 g^* (z_1,\dots, z_n)) := \overline{g \left ( ({\overline{z_1}})^{-1},\dots, ({\overline{z_n}})^{-1}\right )}\, .   
\end{equation}
This operation coincides with complex conjugation if all variables are in $\T$ but preserves analyticity if $g$ extends analytically.
While one may define a pairing \eqref{defquasihilbert} of polynomials more generally for distributions $\mu$ on $\T$, in this paper we
shall mainly be interested in the case of measures.
We say $f$ is left orthogonal to $g$, or $g$ is right orthogonal to $f$, 
if 
$\langle f, g \rangle_{\mu}  = 0$.
We call a polynomial $\Phi$ left or right orthogonal, if 
$\langle \Phi,p\rangle_\mu=0$
or $\langle p, \Phi\rangle_\mu=0$, respectively, 
 for all polynomials  $p$ of degree strictly less than that of $\Phi$.

Define $\mathcal{T}$ to be the class of 
complex probability measures $\mu$ on $\T$ such that, for each $n\in \mathbb{N}_0$, there exist a unique left orthogonal monic polynomial $\Phi_n$ and a unique right orthogonal  monic polynomial $\ti \Phi_n$ of degree $n$. We call the sequences $(\Phi_n)$ and $(\ti \Phi_n)$ the one-sided orthogonal polynomials associated to $\mu$. 
In Lemma \ref{lem:Szego_rec_general}, we show for $\mu\in \mathcal{T}$ 
 the generalized Szeg\H o recurrence relations $\Phi_0=\ti \Phi_0=1$ and,
 for $n\in \mathbb{N}_0$, 
\begin{equation}\label{eq:Szego_recur_intro}
\Phi_{n+1} - z \Phi_n = \overline{F_{n+1}} z^n \ti  \Phi_n ^* \, , \qquad     \ti \Phi_{n+1} - z \ti \Phi_n = \overline{\ti F_{n+1}} z^n \Phi_n ^* \, ,
\end{equation}
for some complex numbers $F_{n+1}$ and $\ti F_{n+1}$ depending on $\mu$.
Define  $\mathcal{T}_+$ to be the class of measures $\mu \in \mathcal{T}$ for which $F_{n} = \ti F_{n}$ for all $n\geq 0$.
In this case, induction with \eqref{eq:Szego_recur_intro} shows $\Phi_n = \ti \Phi_n$ for all $n \geq 0$. 
The class $\mathcal{T}_+$ contains all nonnegative probability measures on $\T$, when the polynomials $(\Phi_n)$ are the well-studied orthogonal polynomials on the unit circle \cite{simon}. More generally, $\mathcal{T}_{+}$ contains all real probability measures, whose associated polynomials were studied in \cite{Landau, Krein_sign}.

Motivated by the $SU(2)$-valued nonlinear Fourier series \cite{tsai}, here we  study the class $\mathcal{T}_{-}$ of measures $\mu \in \mathcal{T}$ for which $F_{n} = -\ti F_{n}$ for all $n\geq 1$.
We have not been able to identify literature where this class  has been studied.
If $\mu \in \mathcal{T}_{-}$, then  \eqref{eq:norm_poly_potential} shows that the monic polynomials $\Phi_n$ and $\ti \Phi_n$ satisfy 
\begin{equation}\label{normPhi}
   \langle \Phi_n,\ti \Phi_n\rangle_\mu   =   \prod\limits_{j=1} ^n (1 + |F_j|^2) >0\, .
 \end{equation}
 Thus one may divide the polynomials $\Phi_n$, $\ti \Phi_n$ 
 by the square root of the left side of \eqref{normPhi} and obtain so-called normalized left and right orthogonal polynomials $\phi_n$, $\ti \phi_n$  with positive leading coefficient and, for all $j,k\in \mathbb{N}_0$, with the Kronecker delta $\delta_{jk}$,
\begin{equation}\label{orthonormal}
    \langle \phi_j, \ti \phi_k \rangle_{\mu} = \delta_{j k} \, .
\end{equation}
In particular, \eqref{orthonormal} shows that the function
\begin{equation}\label{reprokernel}
    K_n(z, \lambda) := \sum_{j=0}^n \ti\phi_j(z ) {\phi_j^*(\lambda)}  
\end{equation}
is a reproducing kernel, since it satisfies
\begin{equation}\label{reproducing}
    \langle f(\cdot) , K_n (\cdot, \lambda ) \rangle_{\mu}  = f(\lambda) 
\end{equation}
for all normalized left orthogonal polynomials $f$ of degree at most $n$ and thus for all polynomials $f$ of degree at most $n$. 
Note that $K$ is analytic in both variables due to the star in \eqref{reprokernel} in place of a complex conjugation often done in the literature.

Our first main result is a Mate-Nevai-Totik \cite{matenevaitotik} type universality theorem for reproducing kernels $K_n$ generated by measures $\mu \in \mathcal{T}_{-}$. Define the Dirichlet kernel
    \[
    D_n (z, \lambda) := \sum\limits_{j=0} ^{n} z^j {\lambda^{-j}}  
    \]
    as the reproducing kernel associated to the uniform probability measure $\frac{d |z|}{2 \pi}$ on $\T$. 
Given $\mu\in \mathcal{T}_-$, let $E(\mu)$ denote the set of all $s\in \T$ such that
$s$ is not in the support of the  singular part of $\mu$, $s$ is a Lebesgue point of the Radon-Nikodym derivative $\frac{w}{2 \pi}$ of $\mu$ with respect to arclength, and in particular $|w(s)| < \infty$.
For $s\in E(\mu)$ and $n\in \mathbb{N}_0$, define
\begin{equation}\label{defineL}
    L(\mu,s,n):=\int\limits_{\T} \min \left \{n+1, \frac{1}{(n+1)|y-s|^2} \right \}  \left | d\mu (y) - w(s) \meas{y} \right |\, ,
    \end{equation}
    where $d |y|$ denotes the arclength measure on $\T$.
 Recall that Lebesgue differentiation implies that for all $s\in E(\mu)$ we have
\begin{equation}
    \lim_{n\to \infty} L(\mu,s,n)=0\, ,
\end{equation}
because the integrand in
\eqref{defineL} is an approximate identity while the measure   in
\eqref{defineL} has Lebesgue density $0$ at $s$.
\begin{theorem}\label{fejerconvergence}
    Let $\mu \in \mathcal{T}_{-}$, $C\geq 2$, $s\in E (\mu)$ and  $n\in \mathbb{N} _0$ with $n\ge 2C$. If we have $|z-s|\le C/n$ and $|\lambda-s|\le C/n$, then
    \begin{equation}\label{eq:KLbound}
        \frac{1}{n+1} \left| \overline{w (s)} K_n(z, \lambda ) -  D_n(z, \lambda ) \right| \leq  e^{30C} L(\mu,s,n) \, .
    \end{equation}
\end{theorem}
See \cite{Bessonov, Gubkin, EiLuSi, EiLuWo} for related universality results. Theorem \ref{fejerconvergence} can also be seen as a nonlinear generalization of the convergence of Fej\'er means, as will be elaborated in the appendix. 
Thanks to  Theorem \ref{fejerconvergence}, we develop in Section \ref{sec:lpfe} a remarkable local representation of the one-sided orthogonal polynomials $(\phi_n)$ near $s \in \T$ by parameters 
\begin{equation}\label{def:A_B_n}
        A_{n,s} := \frac{\phi_n(s)-\phi_n(s\gamma_n)}{2s^n} \, , \qquad 
    B_{n,s} := \frac{\phi_n(s)+\phi_n(s\gamma_n)}{2} \, , 
    \end{equation}
where 
\begin{equation}\label{gamman}
    \gamma_n :=e^{i\pi/n}\ .
\end{equation}
Our remaining results are all proved by careful analysis of these local parameters in three regions, pictured in Figure \ref{fig:fig}. In what follows, if no measure is specified in an integral over $\T$, we take the measure to be normalized arclength $\meas{s}$.
    \begin{theorem}\label{thm:conv_to_AB}
    Let  $\mu \in \mathcal{T}_{-}$ and $s\in E(\mu)$.
For any increasing integer sequence $(n_k)$,
\begin{equation}\label{subshyp1}
     \lim\limits_{k\to \infty} A_{n_k, s} B_{n_k, s} =0\, ,
    \end{equation}
    implies
\begin{equation}\label{subscon1}
     \lim\limits_{k\to \infty} ({\phi_{n_k}^*(s)}\ti \phi_{n_k}(s))^2 = {\overline{w(s)^{-2}}}\, .
    \end{equation}

Assume further that both $\lim\limits_{n\to \infty}|F_n|=0$ and
\begin{equation}\label{convabhyp}
      \lim\limits_{n \to \infty} \int\limits_{\T} \left | \phi_n^* (s)  \ti \phi_n (s) - \frac{1}{\overline{w (s)}} \right | = 0 \, .
      \end{equation} 
    There is a set $ E_0 (\mu)\subset E(\mu)$ of full measure such that  for every $s\in E_0(\mu)$, we have   
\begin{equation}\label{subshyp2}
     \lim\limits_{n\to \infty} ({\phi_{n}^*(s)}\ti \phi_{n}(s))^2 = {\overline{w(s)^{-2}}} \, ,
    \end{equation} 
    implies both
\begin{equation}\label{subscon2}
     \lim\limits_{n\to \infty} {\phi_{n}^*(s)}\ti \phi_{n}(s) = {\overline{w(s)^{-1}}} \, ,
    \end{equation}
    \begin{equation}\label{subscon3}
     \lim\limits_{n\to \infty} A_{n,s} B_{n,s}=0 \, .
    \end{equation}
\end{theorem}

\begin{figure}
    \centering
     \caption{Local parameter regions}
    \label{fig:fig}
\begin{tikzpicture}
\node[draw, circle, color=white] (Z) at (0, 2) {};
    \node[draw, circle, fill=yellow] (A) at (0, 1) { $A, B\sim 1$};
    \node[draw, circle, fill=cyan] (B) at (-2.5, 0) { $A<<1$};
    \node[draw, circle, fill=cyan] (C) at (2.5, 0) { $B<<1$};

    \draw (B) -- (A);
    \draw (C) -- (A);
\end{tikzpicture}
\end{figure}
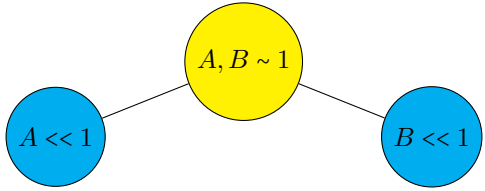

As Lemma \ref{smallablem} shows,
small $A$ has the effect that $ \phi_n^*  \ti \phi_n$ is close to $-\overline{w(s)^{-1}}$, while small $B$ has the effect that $ \phi_n^*  \ti \phi_n$ is close to $\overline{w(s)^{-1}}$. The first part of
Theorem \ref{thm:conv_to_AB} tautologically avoids the yellow region 
via assumption \eqref{subshyp1}, and thus obtains the convergence of $(\phi_n^*  \ti \phi_n)^2$ to $\overline{w(s)^{-2}}$. The yellow region creates large fluctuations of $ \phi_n^*  \ti \phi_n$ on Heisenberg large arcs of $\T$. This region can be suppressed by a milder integral condition \eqref{convabhyp}, as is done in the second part of Theorem \ref{thm:conv_to_AB}, or later in Theorem \ref{thm:lacunary}. The integral assumption \eqref{convabhyp} in Theorem \ref{thm:conv_to_AB}, combined with results of Section \ref{sec:lpre},
together provide a positive distance between the regions of small $A$ and small $B$. As we 
control the regions for all $n$ by the very strong assumption \eqref{subshyp2} over the full sequence of natural numbers,
smallness of $(F_n)$ implies
that asymptotically one can not jump this distance. The sequence therefore stabilizes in one of the regions, and the region of small $B$ is preferred 
given
the convergence \eqref{convabhyp} in mean. We obtain almost everywhere convergence of $ \phi_n^*  \ti \phi_n$ to
$\overline{w(s)^{-1}}$.

In the case $\mu \in \mathcal{T}_+$, there has been recent activity connecting the behavior of roots of the sequence $(\Phi_n)$ of orthogonal polynomials on the unit circle (OPUC, see \cite{simon} for an extensive treatise) to  pointwise almost everywhere boundedness \cite{Poltoratski, bessonovdenisov2021zero}. 
The following theorem shows that condition 
\eqref{subshyp1} is equivalent to the condition
 \begin{equation}\label{subshyp3}
        \liminf_{k\to \infty} n_k\min\{|z-s|: \phi_{n_k}(z)=0\} = \infty
    \end{equation}
on the zeros of the left orthogonal polynomials. The combination of 
Theorems \ref{thm:conv_to_AB} and \ref{zeros theorem} then provides an 
analog of \cite[Theorem 3]{bessonovdenisov2021zero} in the present setting of  measures $\mu \in \mathcal{T}_{-}$, connecting pointwise convergence to the behavior of zeros of orthogonal polynomials.

\begin{theorem}\label{zeros theorem}
    Let  $\mu \in \mathcal{T}_{-}$ and $s\in E(\mu)$. For all $0<\epsilon < 10^{-5}$, there is an $n_0>40$ such that, for all $n>n_0$, we have \eqref{convergence of AB1} below implies \eqref{zeros going to infinity1} below:
\begin{equation}\label{convergence of AB1}
     |A_{n,s} B_{n,s}|\ge  \epsilon \, .
    \end{equation}
\begin{equation}\label{zeros going to infinity1}
        \min\{|z-s|: \phi_{n}(z)=0\} \le 10^4\epsilon^{-1}n^{-1} \, .
    \end{equation}
Moreover,  if in addition $e^{\epsilon^{-1}} > |w(s)|$, then \eqref{zeros going to infinity2}
implies \eqref{convergence of AB2}:
\begin{equation}\label{zeros going to infinity2}
        \min\{|z-s|: \phi_{n}(z)=0\} \le 
  \epsilon^{-1}n^{-1} \, .
    \end{equation}
\begin{equation}\label{convergence of AB2}
     |A_{n,s} B_{n,s}|\ge e ^{-10\epsilon^{-1}}\, .
    \end{equation}
\end{theorem}
A monotone increasing sequence $(n_k)$ is called lacunary if there exists an $L>1$ such that for all $k\in \mathbb{N}_0$ we have
$n_{k+1}\ge Ln_k$.
\begin{theorem}\label{thm:lacunary}
    Let $\mu \in \mathcal{T}_{-}$ and assume the integral condition \eqref{convabhyp} and that $\|(F_n)\|_{\ell^2(\mathbb{N}_0)}$ is finite. 
    If $(n_k)$ is a lacunary sequence, then for almost every $s \in \T$, we have  \[
    \lim\limits_{k \to \infty} { (\phi_{n_k}^* (s)} \ti\phi_{n_k} (s))^2 = {\overline{w (s)^{-2}}}\, .
    \]
    
\end{theorem}
If  $F_n$ is square
summable,
one can  relate measures $\mu \in \mathcal{T}_{-}$ to the $SU(2)$-valued nonlinear Fourier series, similarly 
to how one can relate measures $\mu \in \mathcal{T}_{+}$ and the associated OPUC  to the $SU(1,1)$-valued nonlinear Fourier series \cite{TaoThiele2012}.  For such $(F_n)$, define $\phi_n$ and $\ti \phi_n$ by \eqref{recursion phi} and \eqref{recursion ti phi} and initial condition $\phi_0 = \ti \phi_0 = 1$, and define the functions
\begin{equation}\label{eq:NLFT_defn_intro}
a_n (z):= \frac{\phi_n (z) + \ti \phi_n (z)}{2 z^n} \, , \quad  b_n ^* (z):=  \frac{\phi_n (z) - \ti \phi_n (z)}{2 z^n}\, .
\end{equation}
If  $F:=(F_j)$  is square summable, then $(a_n^*,b_n)$ has an $H^2 (\D) \times H^2 (\D)$ limit $(a,b)$, where the Hardy space $H^p(\D)$
is a space of analytic functions on $\D$ with $L^p$
extensions to $\T$. This limit $(a,b)$ is called the $SU(2)$-nonlinear Fourier series (NLFS) of the sequence $F$, see \cite{tsai} for details on the NLFS.
A celebrated open problem, called the nonlinear Carleson problem, is to determine whether $(a_n,b_n)$ converges almost everywhere to $(a,b)$. 
The paper \cite{Poltoratski} suggests the study of the simpler question of almost everywhere convergence of functionals of $(a_n,b_n)$ such as $|a_n|$.
In the $SU(2)$ setting, convergence of 
$\phi_{n_k}^*\ti\phi_{n_k}$
implies 
convergence of 
$|\phi_{n_k} (s)+\ti\phi_{n_k} (s) |^2 $
by \eqref{eq:ortho_polys_det2},
which implies convergence of $|a|$ by \eqref{eq:NLFT_defn_intro}.
However, in the present paper, we are not able to prove very strong results of convergence of
$\phi_{n_k}^*\ti\phi_{n_k}$ and instead address
the even simpler question of convergence of $(\phi_{n_k}^*\ti\phi_{n_k})^2$,
which can  be expressed
as a functional in $(a_{n_k},b_{n_k})$, as done in Theorem \ref{thm:su2}.
Unfortunately, not all
square summable sequences $F_n$
are covered by the techniques in this paper, as the second part of Theorem \ref{thm:su2} shows. The first part of Theorem \ref{thm:su2} gives  sufficient condition for square summable sequences in terms of the NLFS to provide measures $\mu\in \mathcal{T}_-$. In particular,
Theorem \ref{thm:su2} constructively provides an 
interesting class of measures $\mu \in \mathcal{T}_{-}$ to which all the previous Theorems \ref{fejerconvergence}-\ref{thm:lacunary} apply.

\begin{theorem}\label{thm:su2}
\begin{enumerate}[align=left, leftmargin=0pt, itemindent=1em, labelsep=0.2em, nosep]
    \item \label{thm:su2_part1}
  Let $b\in H^\infty (\D)$ vanish at the origin and satisfy
  \begin{equation}\label{su2hyp1}
      \|b\|_{H^\infty(\D)}<2^{-\frac 12}\, .
  \end{equation}
Let $a^*\in H^2(\D)$
   be outer such that for almost every $s\in \T$ 
 \begin{equation}\label{su2hyp2}
     |a(s)|^2+|b(s)|^2=1\, .
  \end{equation} 
   Then there exists
   $(F_n)\in \ell^2(\mathbb{N}_0)$ so that $(a,b)$ is the associated nonlinear Fourier series of $(F_n)$ and $(\phi_n,\ti \phi_n)$ defined by
\eqref{recursion phi}, \eqref{recursion ti phi} 
are the normalized left and right orthogonal polynomials of a measure $\mu \in \mathcal{T}_-$
satisfying \eqref{convabhyp}. Moreover, let $(n_k)$ be a lacunary sequence. For almost every $s\in \T$,
 \begin{equation}\label{su2con}
     \lim_{k\to \infty}
(a_{n_k}^*+b_{n_k})^2(a_{n_k}-b_{n_k}^*)^2(s)=(a^*+b)^2(a-b^*)^2(s)\, .
\end{equation} 
\item \label{thm:su2_part2}
Requirement \eqref{su2hyp1} is sharp, namely there exists an $(F_n)\in \ell^2(\mathbb{N}_0)$ whose associated NLFS $(a,b)$ satisfies $a^* \in H^2 (\D)$ is outer and
\[
\|b\|_{H^{\infty} (\D)} = 2^{- \frac 1 2} \, ,
\]
but the monic polynomials
$(\Phi_n,\ti \Phi_n)$ defined  by \eqref{eq:Szego_recur_intro} are not left and right orthogonal polynomials for any measure $\mu\in \mathcal{T}_-$.
\end{enumerate}
\end{theorem}



Examples of $b$ and $a$ as in the first part of Theorem \ref{thm:su2} are easy to come by, as constructing outer $a^*$ for given modulus on the boundary is standard \cite[Theorem 10]{QSP_NLFA}.
This class of $(a,b)$ and square summable sequences arose in recent
studies on quantum signal processing \cite{QSP_NLFA,QSP_NLFA_2}, see also \cite{low2017,Linlin}
for more on quantum signal processing.

Results related to but not solving the nonlinear Carleson problem can be found in
\cite{ChristKiselev,christkiselev2002, Oberlin2012, diogo, KovacDiogo, MuscaluTaoThiele,MuscaluCounterexample,Poltoratski}.  For a discussion of some open problems we refer to the survey \cite{diogosurvey}. This present paper is initially motivated by such convergence questions in the $SU(2)$ model and in particular echoes the approach of \cite{Poltoratski} for the $SU(1,1)$ model.

Similar to how Krein systems are the continuous analogs of OPUC \cite{Denisov}, our results for left and right polynomials  likely have a continuous analog. We also note that Krein systems are further connected to de Branges canonical systems \cite{romanov, Remling,Poltoratski2015ToeplitzAT}, which may be a helpful framework by which to understand our left and right polynomials.

\ \\

 {\bf Acknowledgments.}
 M.~A. and C.~T. were funded by the Deutsche Forschungsgemeinschaft (DFG, German Research Foundation) under Germany's Excellence Strategy -- EXC-2047/1 -- 390685813 as well as SFB 1060.

During an Arbeitsgemeinschaft at the Oberwolfach Institute in October 2024, Gennady Uraltsev suggested showing the nonlinear Carleson theorem along lacunary sequences. Our efforts towards this goal resulted in Theorem \ref{thm:lacunary}. The authors thank Gennady Uraltsev for many inspiring discussions, and the Oberwolfach Institute for hosting that event.


\section{Polynomials generated by a complex measure}\label{polynomials generated by a complex measure}
Let $\mu$ be a complex probability measure on $\T$, i.e., a complex measure with
\[
\int\limits_{\T} d \mu = 1 \, .
\]
Let $f$ and $g$ be polynomials in $z \in \T$.
Recalling definitions \eqref{defquasihilbert} and 
\eqref{stardef},  note 
\begin{equation}\label{mumubar}
    \langle f, g \rangle_{\mu}= \overline{\langle g, f \rangle_{\overline{\mu}}} \, ,
\end{equation} 
\begin{equation} \label{switchstar}
    \langle f, g \rangle_{\mu}=\langle g^*, f^* \rangle_{\mu} \, .
\end{equation}
 By \eqref{mumubar}, a polynomial is right orthogonal with respect to $\mu$ precisely if it is left orthogonal with respect to $\overline{\mu}$. 
 \begin{lemma}\label{unique}
    Let $\mu$ be a complex probability measure on $\T$ and assume that  for each $n\in \mathbb{N}_0$ there exists a left orthogonal monic polynomial $\Phi_n$ of degree $n$. The following statements are equivalent:
\begin{enumerate}
    \item \label{item:inner_prod_nonzero} For each $n\in \mathbb{N}_0$, 
\begin{equation}\label{unique1}
    \langle\Phi_n, z^n\rangle_\mu\neq 0\, .
    \end{equation}
\item \label{item:unique_poly} For each $n\in \mathbb{N}_0$, $\Phi_n$ is the unique degree $n$ left orthogonal monic polynomial.
\item \label{item:det_nonzer}  For each $n \in \mathbb{N}_0$, the matrix determinant
\begin{equation}\label{unique90}
\Delta_n := \det \begin{pmatrix}
    c_0 & c_{-1} & \ldots &  c_{-n} \\
   c_1 & c_0 & \ldots &  c_{-(n-1)}\\
    \vdots & \vdots & \ddots &  \vdots \\ 
   c_{n}  & c_{n-1} & \ldots &  c_0 
\end{pmatrix}
\end{equation}
is nonzero, where $c_{j} := \int\limits_{\T} z^j d\mu$ denotes the $j$-th moment of $\mu$.
\end{enumerate}
In particular, if any of the above hold, then for all $n \geq 1$, we have Heine's formula 
\begin{equation}\label{eq:Heine}
\Phi_n (z) = \frac{1}{\Delta_{n-1}} \det \begin{pmatrix}
   z^n & z^{n-1} & z^{n-2} & \ldots & z & 1 \\
   c_{1} & c_0 & c_{-1} & \ldots & c_{-(n-2)} & c_{-(n-1)} \\
   c_2 & c_1 & c_0 & \ldots & c_{-(n-3)} & c_{-(n-2)}\\
   \vdots & \vdots & \vdots & \ddots & \vdots & \vdots \\ 
   c_n & c_{n-1}  & c_{n-2} & \ldots & c_1 & c_0 
\end{pmatrix} \, .
\end{equation}
    \end{lemma}
\begin{proof}
    Let $\mu$, $\Phi_n$ be given as in the lemma.
We first show that \eqref{item:inner_prod_nonzero} and \eqref{item:unique_poly} are equivalent, by showing each fails if and only if and the other fails. Assume that \eqref{item:unique_poly} fails, i.e., there is $n\in \mathbb{N}_0$ such that
there exists a left orthogonal monic polynomial $\Psi_n\neq \Phi_n$.
We may assume $n$ is minimal for which such  $\Psi_n$ exists. 
 The nonzero polynomial $\Phi_n-\Psi_n$ is left orthogonal
to all polynomials of degree less than $n$. As difference of two monic polynomials, 
it has degree strictly less than $n$, call it $j$. Dividing by its highest coefficient, we obtain a left orthogonal monic polynomial  of degree $j$. By minimality of $n$, this polynomial is equal to $\Phi_j$. We conclude $\langle \Phi_j,z^j \rangle_\mu=0$ as $\Phi_n-\Psi_n$ is left orthogonal to all polynomials of degree less than $n$. Thus \eqref{item:inner_prod_nonzero} fails. 

Now assume that \eqref{item:inner_prod_nonzero} fails, i.e., there is some $n$ for which we have $\langle\Phi_n, z^n \rangle_\mu= 0\, .
$
Then $\Phi_n$ is left orthogonal
to all polynomials of degree $n$, since from
any such polynomial we can subtract a multiple of $z^n$ to obtain a polynomial of lower degree. Then $\Phi_{n+1}+\Phi_n$ 
is a left orthogonal
monic polynomial of degree $n+1$ different from $\Phi_{n+1}$, meaning  \eqref{item:unique_poly} fails. Thus \eqref{item:unique_poly} and \eqref{item:inner_prod_nonzero} are equivalent.

We now show \eqref{item:unique_poly} and \eqref{item:det_nonzer} are equivalent. We assume first that \eqref{item:det_nonzer} holds. Noting that $\Phi_0 = 1$ is unique monic polynomial of degree $0$, let $n \geq 1$. Let 
\begin{equation}\label{unique100}
   \Phi_n ^{k}(z) = z^n + \sum\limits_{j=0}^{n-1} x_j ^k z^j \, , 
\end{equation}
denote two left orthogonal polynomials for $k=1, 2$. Then 
\begin{equation}\label{unique101}
\langle \Phi_n ^k, z^j \rangle_{\mu} = 0 \, , \qquad 0 \leq j\leq n-1 \, .
\end{equation}
Expanding in terms of $(x_j ^k )$, this last condition is equivalent the linear system 
\begin{equation}\label{unique102}
\begin{pmatrix}
   c_0 & c_1 & \ldots & c_{n-1}  \\
   c_{-1} & c_0 & \ldots & c_{n-2}  \\
   \vdots & \vdots & \ddots & \vdots  \\
   c_{-(n-1)} & c_{-(n-2)} & \ldots & c_0  
\end{pmatrix} \begin{pmatrix}
   x_0  \\ x_1  \\ \vdots \\ x_{n-1}   
\end{pmatrix} = - \begin{pmatrix}
    c_n \\ c_{n-1} \\  \vdots \\ c_1
\end{pmatrix}
\end{equation}
having solutions $(x_j ^1)$ and $(x_j ^2)$. 
But the matrix on the left of \eqref{unique102} has nonzero determinant, and hence trivial kernel. Hence $(x_{j} ^1) = (x_j ^2)$, meaning $\Phi_n ^1 = \Phi_n ^2$. Thus \eqref{item:unique_poly} holds.

Now assume \eqref{item:unique_poly} holds, and let $n \geq 1$. Then writing our given left orthogonal polynomial 
\begin{equation}\label{unique103}
   \Phi_n (z) = z^n + \sum\limits_{j=0} ^{n-1} x_j z^j \, , 
\end{equation}
 we have that \eqref{unique102} has a solution. Then $\Delta_{n-1} \neq 0$ if and only if the matrix on the left of \eqref{unique102} has trivial kernel, which holds if and only if \eqref{unique102} has at most one solution. Let $(x_j ^1)$ and $(x_j ^2)$ denote two solutions. Then the polynomials defined in \eqref{unique100} must both be left orthogonal polynomials. By assumption, they must be equal, and hence $(x_j ^1) = (x_j ^2)$. Thus $\Delta_{n-1} \neq 0$. Thus \eqref{item:det_nonzer} holds.
 
We now prove \eqref{eq:Heine} holds. Let $\Psi_n$ denote the right side of \eqref{eq:Heine}, which is well-defined by \eqref{item:det_nonzer}. By cofactor expansion, $\Psi_n$ is monic of degree $n$. And as for left orthogonality, we note that
\[
\langle \Psi_n, z^j \rangle_{\mu} = \frac{1}{\Delta_{n-1}} \det \begin{pmatrix}
   c_{n-j} & c_{n-1-j} & c_{n-2-j} & \ldots & c_{-(j-1)} & c_{-j} \\
   c_{1} & c_0 & c_{-1} & \ldots & c_{-(n-2)} & c_{-(n-1)} \\
   c_2 & c_1 & c_0 & \ldots & c_{-(n-3)} & c_{-(n-2)}\\
   \vdots & \vdots & \vdots & \ddots & \vdots & \vdots \\ 
   c_n & c_{n-1}  & c_{n-2} & \ldots & c_1 & c_0 
\end{pmatrix}, 
\]
which vanishes if $0 \leq j \leq n-1$, because then the determinant is of a matrix with a repeated row. Thus $\Psi_n$ is a monic left orthogonal polynomial of degree $n$, and by \eqref{item:unique_poly}, it must equal $\Phi_n$. 
\end{proof}

\begin{lemma}\label{lem:Szego_rec_general}
    Let $\mu \in \mathcal{T}$. For each $n\in \mathbb{N}_0$, there exist unique constants $F_{n+1}, \ti F_{n+1} \in \C $ such that we have the Szeg\H o recurrence
    \eqref{eq:Szego_recur_intro}.
    \end{lemma}

\begin{proof}

 Because each $\ti \Phi_n$ is monic,
induction shows that for each $n\in \mathbb{N}_0$ we have
\begin{equation}\label{largespan}
\operatorname{Span} \{1, z, \ldots, z^{n} \}=
\operatorname{Span} \{ \ti \Phi_0,  \ti \Phi_1 , \ldots,   \ti \Phi_{n}  \} \, . \end{equation}
Applying the star operation \eqref{stardef} to \eqref{largespan} for $n-1$ and multiplying by $z^n$ gives
\begin{equation}\label{smallspan}
\operatorname{Span} \{z^n, z^{n-1}, \ldots, z^{1} \}=
\operatorname{Span} \{z^n \ti \Phi_0 ^*, z^n \ti \Phi_1 ^*, \ldots, z^{n}  \ti \Phi_{n-1} ^* \}.    
\end{equation}
Because both $\Phi_{n+1}$ and $z \Phi_n$ are monic polynomials of degree $n+1$,
  their difference is a polynomial of degree $n$. Let $F_{n+1}$ denote the coefficient of the constant term of this polynomial,
then
\begin{equation}\label{expanddifference}
\Phi_{n+1} - z \Phi_n -\overline{F_{n+1}} z^n \ti \Phi_n ^* = \sum\limits_{j=0}^{n-1} d_j z^n \ti \Phi_j ^* \, ,\end{equation}
because the left side is in the space  \eqref{smallspan}. 
Let $0 \leq \ell \leq n-1$ and pair \eqref{expanddifference}
from the right with $z^n \Phi^*_\ell$ to obtain
\begin{equation}\label{expand1}
\langle \Phi_{n+1},z^n\Phi^*_\ell\rangle_\mu
- \langle z \Phi_n ,z^n\Phi^*_\ell\rangle_\mu -\overline{F_{n+1}} \langle  z^n \ti \Phi_n ^* ,z^n\Phi^*_\ell\rangle_\mu = \sum\limits_{j=0}^{n-1} d_j 
\langle z^n \ti \Phi_j ^*,z^n\Phi^*_\ell\rangle_\mu  \, ,\end{equation}
The first term on the left side of \eqref{expand1} vanishes
because  $\Phi_{n+1}$ is left orthogonal and $z^{n}\Phi^*_\ell$ is a polynomial of degree less than $n+1$.  
Similarly, the second term vanishes because we can divide both factors in the pairing by $z$
and $z^{n-1}\Phi^*_\ell$
is a polynomial of degree less than $n$. The pairing in the third term can be rewritten with \eqref{switchstar} as  
\begin{equation}\label{expand2}
 \langle  z^n \ti \Phi_n ^* ,z^n\Phi^*_\ell\rangle_\mu
 =\langle   \ti \Phi_n ^* ,\Phi^*_\ell\rangle_\mu
 =\langle  \Phi_\ell,
 \ti \Phi_n \rangle_\mu=0\, ,
\end{equation}
where we have used right orthogonality of $\ti \Phi_n$.
By the analogous computation as \eqref{expand2}, most terms on the right side of \eqref{expand1} vanish and we obtain
\begin{equation}\label{expand3}
 0
 = d_\ell \langle  \Phi_\ell,
 \ti \Phi_\ell \rangle_\mu\, .
\end{equation}
By Lemma \ref{unique}, we obtain $d_\ell=0$. As $0\le \ell\le n-1$ was arbitrary,  identity \eqref{expanddifference} becomes the first Szeg\H o relation in \eqref{eq:Szego_recur_intro}.
The second relation is shown analogously, switching the roles of left and right. Uniqueness of $F_n$ and $F_{n+1}$ follows because the polynomials on the right side of the 
Szeg\H o relations \eqref{eq:Szego_recur_intro} are nonzero.

\end{proof}

In what follows, we employ the convention that a product indexed by the empty set equals $1$.
\begin{lemma}\label{lem:phiortho}
If $\mu\in \mathcal{T}_-$, then for all $n \geq 0$, we have 
\begin{equation}\label{eq:norm_poly_potential}
\langle \Phi_n , \ti \Phi_n \rangle_{\mu}  =  \prod\limits_{j=1}^n (1 + |F_j|^2)  \, . 
\end{equation}
 
\end{lemma}
\begin{proof}
    From the Szeg\H o recurrence \eqref{eq:Szego_recur_intro}, it follows that for  $n \geq 0$, we have
\begin{equation}\label{eq:pre_Szego_recursion}
\begin{pmatrix}
    \frac{\Phi_{n+1}}{z^{n+1}} & - \left ( \frac{ \ti \Phi_{n+1}}{z^{n+1}} \right )^* \\ \frac{\ti \Phi_{n+1}}{z^{n+1}} & \left ( \frac{\Phi_{n+1}}{z^{n+1}} \right )^*
\end{pmatrix} = \begin{pmatrix}
    \frac{\Phi_{n}}{z^{n}} & - \left ( \frac{\ti \Phi_{n}}{z^{n}} \right )^* \\ \frac{\ti \Phi_{n}}{z^{n}} & \left ( \frac{\Phi_{n}}{z^{n}} \right )^*
\end{pmatrix} \begin{pmatrix}
    1 &  F_{n+1} z^{n+1}\\
    - \overline{F_{n+1}} z^{-(n+1)} & 1 
\end{pmatrix} \, .
\end{equation}
Taking determinants and applying induction with 
the initial condition $\Phi_0 = \ti \Phi_0 =1$, which follows from the monic assumption, gives for  $n \geq 0$,
\begin{equation}\label{eq:det_polys}
\Phi_n \Phi_n ^* + \ti \Phi_n (\ti \Phi_n) ^* = 2 \prod\limits_{j=1}^n (1 + |F_j|^2) \, .
\end{equation}
Integrating \eqref{eq:det_polys} with respect to the probability measure $\mu$ on $\T$ yields
\begin{equation}\label{eq:Phi_det_integrate}
\langle \Phi_n ,  \Phi_n \rangle_{\mu}+
\langle \ti \Phi_n , \ti \Phi_n \rangle_{\mu}
 = 2 \prod\limits_{j=1}^n (1 + |F_j|^2) \, .
\end{equation}
But $\Phi_n$ and $\ti \Phi_n$ are left and right orthogonal, respectively, to the polynomial $\ti \Phi_n - \Phi_n$ because it has degree at most $n-1$, and therefore
\begin{equation}\label{eq:norms_Phi_inner_prod}
\langle \Phi_n ,  \Phi_n \rangle_{\mu} = \langle \Phi_n , \ti \Phi_n \rangle_{\mu}  = 
\langle \ti \Phi_n , \ti \Phi_n \rangle_{\mu} 
\end{equation}
Identities \eqref{eq:Phi_det_integrate} and \eqref{eq:norms_Phi_inner_prod} yield
\eqref{eq:norm_poly_potential}, which proves the lemma.
\end{proof}
With Lemma \ref{lem:phiortho} and left-orthogonality of $\Phi_n$, we have
\[
\langle \Phi_n, \ti \Phi_n \rangle_{\mu} =  \langle \Phi_n, z^n \rangle_{\mu} \neq 0 \, .
\]
Thus we define the normalized polynomials
 \begin{equation}\label{defphi}
     \phi_n = \frac{\Phi_n}{\sqrt{\langle \Phi_n , \ti \Phi_n \rangle_{\mu}}} \, ,  \qquad \ti \phi_n = \frac{\ti \Phi_n}{\sqrt{\langle \Phi_n , \ti \Phi_n \rangle_{\mu}}}\, .
 \end{equation}
 They inherit from $\Phi_n$, $\ti \Phi_n$ and Lemma \ref{lem:phiortho} the orthonormality relation \eqref{orthonormal} thanks to the normalization of $\mu$ as probability measure, and using in addition 
 Lemma \ref{lem:Szego_rec_general}, we see that they also satisfy the recursion relations $\phi_0=\ti \phi_0=1$
  and 
\begin{equation}\label{recursion phi}
    \phi_{n+1} = \frac{1}{\sqrt{1+|F_{n+1}|^2}} \left( z\phi_n + z^n \overline{F_{n+1}} \ti \phi_n^* \right) \, ,
\end{equation}
\begin{equation}\label{recursion ti phi}
    \ti \phi_{n+1} = \frac{1}{\sqrt{1+|F_{n+1}|^2}} \left( z\ti \phi_n - z^n \overline{F_{n+1}} \phi_n ^* \right) \, .
\end{equation}
Rewriting the recursion \eqref{recursion phi}, \eqref{recursion ti phi} as a matrix product gives
\begin{equation}\label{phirecursion}
   \phimat_{n+1} (z) =\phimat_n (z) \zmat^{1-n} \fmat_{n+1} \zmat^{1+n}
\end{equation}
with the matrices
\begin{equation}
    \phimat_n (z) =2^{-\frac 12}
   \begin{pmatrix}
       \phi_{n} (z) & -\ti {\phi_{n}}^* (z) \\
       \ti \phi_{n} (z) & \phi_{n}^* (z)
    \end{pmatrix}
\end{equation}
\begin{equation}
    \fmat_n=\frac 1{\sqrt{1+|F_n|^2}}\begin{pmatrix} 1 & F_n\\ -\overline{F_n } & 1\end{pmatrix}\ ,
\end{equation}
\begin{equation}\label{defzmat}
    \zmat = \begin{pmatrix} z^{\frac 1 2} & 0\\ 0 & z^{-\frac 12} \end{pmatrix}\ .
\end{equation}
The square root of $z$ in the definition of $\zmat$ is to be understood formally in the sense that
$\zmat$ will always appear 
with an accumulated even power
in a product, so that the resulting product can be rewritten
involving only integer powers of $z$.

With the Hadamard gate
\begin{equation}
    \hmat= 2^{-\frac 12}\begin{pmatrix}
       1 & -1\\
       1 & 1
    \end{pmatrix}\, ,
\end{equation}
we have the initial condition $\phimat_0=\hmat$ and thus can
solve the recursion \eqref{phirecursion} as
\begin{equation}\label{phiproduct}
    \phimat_n=\hmat \prod_{0<j\le n} ^{\curvearrowright}\zmat^{2-j} \fmat_j\zmat^{j}=\hmat \left ( \prod_{0<j\le n} ^{\curvearrowright} \zmat \fmat_j \right  )\zmat^{n} \, ,
\end{equation}
where we use the directed product notation
\[
\prod_{a< j\le b} ^{\curvearrowright} M_{j} :=  M_{a+1}\ldots M_{b-1} M_b \, .
\]
Because for $z\in \T$, we have $\hmat, \zmat, \fmat_j$ are all unitary with determinant $1$, then by \eqref{phiproduct} the product $\phimat_n$ must be as well. Taking the determinant of $\phimat_n$ thus yields for $s\in \T$, 
\begin{equation} \label{eq:ortho_polys_det2}
    |\phi_n (s) |^2+|\ti\phi_n (s) |^2=2\, .
\end{equation}
We note the following Plancherel type inequality
\begin{lemma}\label{lem:plancherel}
Let $l,m\in \mathbb{N}_0$ with $l<m$. We have   
\begin{equation}\label{plancon}
-2\int\limits_\T \log(\frac 12|\phi_l^*\phi_m+\ti \phi_l^* \ti \phi_m|)
\le \sum_{l<j\le m} \log(1+|F_j|^2)\, .
\end{equation}

\end{lemma}
\begin{proof}
    We have with \eqref{phiproduct}
\begin{equation}\label{plan1}
\phimat_l^{-1}\phimat_m=\zmat^{-l}\left ( \prod_{l<j\le m} ^{\curvearrowright} \zmat \fmat_j \right )\zmat^m\, .
    \end{equation}
    The upper left entry 
\begin{equation}\label{u(z)=}
     u(z):=\frac 12 \left(\phi_l^*\phi_m+\ti \phi_l^* \ti \phi_m\right)(z)
\end{equation}    
   of the left side of 
\eqref{plan1} is a linear combination of monomials $z^k$ for $k \in \ZZ$. Counting powers of $z$
on the right side of \eqref{plan1}, each such monomial $z^k$ is a product of a factor $z^{-\frac l2}$, then $m-l$ factors of which each is  $z^{1/2}$ or $z^{-1/2}$ and then
one factor of $z^{m/2}$.
The lowest possible power $k$ is then $0$, meaning $u$ is a polynomial in $z$. The coefficient in front of $z^0$ can have only 
second row factors of the matrices $\zmat \fmat_j$
and thus is the product of lower right entries of the matrices $\zmat \fmat_j$. 
Evaluating $u$ at $0$ to obtain the constant term, we
have
\begin{equation}
    \frac 12 \left(\phi_l ^* \phi_m+\ti \phi_l ^* \ti \phi_m \right)(0)
=\prod_{l<j\le m} (1+|F_j|^2)^{-\frac 12}\, .
\end{equation}
Using that the 
polynomial  is analytic on $\mathbb{D}$ and 
thus  the logarithm of its modulus is subharmonic, we obtain
\begin{equation}\label{plan2}
\int\limits_\T \log(\frac 12 |\phi_l^*\phi_m+\ti \phi_l^* \ti \phi_m|)
\ge 
\log(\frac 12 |\phi_l^*\phi_m+\ti \phi_l^* \ti \phi_m|)(0)=
-\frac 12 \sum_{l<j\le m} \log(1+|F_j|^2)\, .
\end{equation}
This proves \eqref{plancon} and hence the lemma.
\end{proof}

\begin{lemma}[Christoffel-Darboux formula]\label{christoffeldarboux}
For all $n\in \mathbb{N}_0$ and all $\lambda \neq 0$, we have
\begin{equation}\label{eq:CD}
        (1-z \lambda^{-1}) K_n(z, \lambda) =
     { z^{n+1} \lambda}^{-n-1} \phi_{n+1}^*(z) {\ti \phi_{n+1}(\lambda)}    - \ti\phi_{n+1}(z) {\phi_{n+1}^*(\lambda)} \, .
    \end{equation}
\end{lemma}

\begin{proof}
Define for a complex parameter $\lambda$ the matrices
\begin{equation}
    \qmat:=\begin{pmatrix}
       0 & 1\\
       0 & 0
    \end{pmatrix}\, , \ 
\Lambda := \begin{pmatrix}
\lambda^{\frac 1 2} & 0 \\ 0 & \lambda^{ - \frac 1 2}
\end{pmatrix} \, ,
\end{equation}
where the remarks after the definition \eqref{defzmat} of $\zmat$ also apply analogous to $\Lambda$. Define 
\begin{equation}\label{kn1}L_n(z, \lambda)=2\lambda^{-\frac{n+1}2}z^{\frac{n+1}2}
\tr(\Lambda^{n+1}\phimat_{n+1}^{-1}(\lambda)\qmat \phimat_{n+1}(z) \zmat^{-n-1})
        \ .
\end{equation}  
The innermost triple product of matrices in \eqref{kn1} is
a tensor of the first column of $\phimat^{-1}_{n+1}(\lambda)$ with the second row of $\phimat_{n+1}(z)$ and thus 
\begin{equation}\label{kn10}
    \phimat_{n+1}^{-1}(\lambda)\qmat \phimat_{n+1}(z)=\frac 12\begin{pmatrix}
        \phi_{n+1}^*(\lambda)\ti \phi_{n+1}(z) & { \phi_{n+1}^*(\lambda)}\phi_{n+1}^*(z) \\
    -\ti \phi_{n+1}(\lambda) \ti \phi_{n+1}(z) & -{\ti \phi_{n+1}(\lambda)}\phi_{n+1}^*(z)
    \end{pmatrix}\, .
\end{equation}
With that we compute
$L_n(z, \lambda)$ to be the negative of the right side of \eqref{eq:CD}, and so we will be done once we verify that $L_n(z, \lambda)$ also coincides with the left side of \eqref{eq:CD}.
Note that $L_{-1}( z, \lambda)=0$.
Peeling off one factor in the product \eqref{phiproduct}, we observe
\begin{equation}\label{kn2}
    \Lambda^{n+1}\phimat_{n+1}^{-1}(\lambda)\qmat \phimat_{n+1}(z) \zmat^{-n-1}=
    \fmat_{n+1}^{-1}\Lambda^{-1}\Lambda^{n}\phimat_{n}^{-1}(\lambda)Q \phimat_{n}(z) \zmat^{-n}\zmat\fmat_{n+1}\, .
\end{equation}
  As the trace is invariant under conjugation by a matrix, we obtain\begin{equation}\label{kn3}L_n(z, \lambda)=2\lambda^{-\frac{n+1}2}z^{\frac{n+1}2}
    \tr(\Lambda^{n-1}\phimat_{n}^{-1}(\lambda)\qmat \phimat_{n}(z) \zmat^{-n+1})
        \ .
\end{equation}  
Comparing \eqref{kn3} with $n+1$ and \eqref{kn1} with $n$, we compute for $n\ge -1$
\begin{equation*}
    L_{n+1}(z, \lambda)-L_n(z, \lambda)=
\end{equation*}    
\begin{equation*} 
2 { \lambda^{-\frac{n+2}2}z^{\frac{n+2}2}}{\tr(\Lambda^{n}\phimat_{n+1}^{-1}(\lambda)\qmat \phimat_{n+1}(z) \zmat ^{-n})}   
\end{equation*}
\begin{equation*}
-2
{\lambda^{-\frac{n+1}2}z^{\frac{n+1}2}}
{\tr(\Lambda^{n+1}\phimat_{n+1}^{-1}(\lambda)\qmat \phimat_{n+1}(z) \zmat ^{-n-1})}
\end{equation*}
\begin{equation}\label{kn7}
=2
{\tr\left(
\left(
{\lambda^{-\frac{n+2}2}z^{\frac{n+2}2}} \operatorname{Id}-
{\lambda^{-\frac{n+1}2}z^{\frac{n+1}2}}\Lambda Z^{-1}
\right)
\Lambda^{n}\phimat_{n+1}^{-1}(\lambda)\qmat \phimat_{n+1}(z) \zmat ^{-n}\right)}
\end{equation}
by linearity of the trace and invariance of the trace under conjugation. 
Finally, because the diagonal matrix
\[
\lambda^{- \frac 1 2 } z^{\frac 1 2} \operatorname{Id} -  \Lambda \zmat^{-1}
\]
vanishes at its bottom right entry, then 
\eqref{kn7} becomes  the upper left entry of
\begin{equation}\label{kn8}
2\left(
{\lambda^{-\frac{n+2}2}z^{\frac{n+2}2}}-
{\lambda^{-\frac{n}2}z^{\frac{n}2}}
\right)
\Lambda^{n}\phimat_{n+1}^{-1}(\lambda)\qmat \phimat_{n+1}(z) \zmat ^{-n}
\end{equation}
which, by \eqref{kn10}, is equal to
\begin{equation}\label{kn80}
\left(
{\lambda^{-1}z-1}
\right)
\phi_{n+1} ^* (\lambda)\ti \phi_{n+1} (z) \, .
\end{equation}
By telescoping and then \eqref{reprokernel}, we obtain
\begin{equation} \label{kn9}
    L_n(z, \lambda)=\sum\limits_{\ell=0}^{n} L_{\ell}(z, \lambda) - L_{\ell-1}(z, \lambda)
    =\sum\limits_{\ell=0}^{n} (\lambda^{-1} z-1) \phi_{\ell} ^* (\lambda) \ti \phi_{\ell} (z)= \left(\lambda^{-1}z-1
\right)K_n(\lambda , z)\, ,
\end{equation}
which is the left side of \eqref{eq:CD}.
\end{proof}


The relation \eqref{eq:NLFT_defn_intro} can be written as
\begin{equation}\label{eq:polys_full_prod_formula}
    \gmat_n=\hmat^{-1} \phimat_n \zmat^{-2n} 
\end{equation}
with the matrix
\begin{equation}\label{define}
    \gmat_n=\begin{pmatrix}
    a_n & b_n \\ - b_n^* & a_n^*
\end{pmatrix}
\end{equation}
The product representation \eqref{phiproduct} then gives
\begin{equation}\label{gproduct}
    \gmat_n= \left ( \prod_{0<j\le n} ^{\curvearrowright}\zmat \fmat_j \right ) \zmat^{-n} \, ,
\end{equation}
or equivalently
\begin{equation}\label{gproduct1}
    \gmat_n= \prod_{0<j\le n} ^{\curvearrowright}\zmat ^j\fmat_j\zmat^{-j} \, .
\end{equation}
The latter is the representation of the $SU(2)$- nonlinear Fourier series in \cite{tsai},
\begin{equation}\label{tsai}
     G_n=\prod_{0<j\le n} ^{\curvearrowright} \frac{1}{\sqrt{1 + |F_{j}|^2}}\begin{pmatrix}
    1 & F_{j} z^{j} \\ -\overline{ F_{j}} z^{-j} & 1
\end{pmatrix}\, .
\end{equation}

In the proof of Theorem \ref{thm:su2}  we will refer to facts about the $SU(2)$-valued NLFS in  \cite{tsai, QSP_NLFA}, 
such as if $(F_n)\in \ell^2$, then 
$G_n$ has an entry-wise limit 
\[
G = \begin{pmatrix}
    a & b \\ -b^* & a^*
\end{pmatrix}
\]
in the sense of $L^2(\T)$ which turns out also in $H^2(\D)$
for the second column of $G_n$. We will also use that
\begin{equation}\label{eq:a_positive}
a_n ^* (0) \, , \qquad a^* (0) >0 \, .
\end{equation}

\section{Proof of Theorem \ref{fejerconvergence}}

Let $\mu \in \mathcal{T}_{-}$
and let $C\ge 2$. In what follows, $B(s,r)$ denotes the ball with radius $r$ and center $s$.
\begin{lemma}\label{lem:poly_bd_gen}
Let $n\ge 1$. The functions $f,g,h,k$ defined by
\begin{equation}\label{eq:poly_bds_gen}
f(z)=\phi_n(z) \, , \quad  g(z)=\tilde{\phi_n}(z) \, , \quad h(z)= z^n \phi_n^*(z) \, , \quad  k(z) = z^n \tilde{\phi_n} ^* (z) \, , 
\end{equation}
are polynomials of degree at most $n$
 and bounded by $\sqrt{2}$ on $\D$ and by $e^{2C}$ on the ball $B(0, 1 + C/n)$. 
\end{lemma}

\begin{proof}
Because $\phi_n$ and $\ti \phi_n$ are polynomials of degree $n$, then $f,g, h, k$ are polynomials of degree at most $n$. 
By \eqref{eq:ortho_polys_det2}, the polynomials 
    $f, g, h, k$
 are bounded by $\sqrt{2}$ for $z$ on $\T$ and by the maximum
principle also for $z$ in the unit disc $\D$.
Now assume  $1 \leq |y| \leq 1 + \frac{C}{n}$. Then, with $z=\overline{y^{-1}} \in \D$, and using the bound of $|h(z)| \leq \sqrt{2}$, we have
\[
|f(y)| = |\phi_n(y)|= |\phi_n^*(z)| =  |y^n z^n \phi_n^*(z)|
 = |h(z)y ^n |\leq \sqrt{2}|y|^n \leq \sqrt{2} \left ( 1 + \frac{C}{n} \right )^n \leq  e^{2C} \, .
\]
This along with a similar argument for $g, h, k$, using the bound of $k$, $f$, $g$ by $\sqrt{2}$ on $\mathbb{D}$,   respectively, completes the proof of the lemma.
\end{proof}

\begin{lemma}\label{lem:kbounds}
    Let $s\in \T$
     and let $n  \geq 2 C$. For all $x \in B(0, 1 + C/n)$ and $y\in B(s,C/n)$,  we have the estimate  \begin{equation}\label{kernel estiamte K}
        |K_n(x, y)| \leq e^{12C} \min (n+1,|s-x|^{-1}) \, . 
    \end{equation}
    
    \end{lemma}
    \begin{proof}
     The representation \eqref{reprokernel} of $K_n$ as sum of $n+1$ products of normalized one-sided orthogonal polynomials and Lemma \ref{lem:poly_bd_gen} give the estimate 
\begin{equation}\label{kernel estiamte K first}
        |K_n(x, y)| \leq  e^{4C} (n+1)\, . 
    \end{equation}
To verify \eqref{kernel estiamte K}, it remains to show
\begin{equation}\label{kernel estiamte K 2}
        |K_n(x, y)| \leq   {e^{12C}}{|s-x|^{-1}} \, .  
    \end{equation}
    If 
    \[
    |s-x| \leq 2C n^{-1} \, ,
    \]
    then \eqref{kernel estiamte K first} implies \eqref{kernel estiamte K 2}. Thus it suffices to show \eqref{kernel estiamte K 2} when  
\begin{equation}\label{cd3}
 |s-x|\ge {2C}n^{-1} \, . 
\end{equation}
  The Christoffel-Darboux formula \eqref{eq:CD} gives for $K_n(x,y)$ the estimate
\begin{equation}\label{darbouxxy}
    |y^n (y-x)| | K_n (x,y)| = |x^{n+1}\ti\phi_{n+1}(y) \phi_{n+1}^*(x)  - y^{n+1}\ti\phi_{n+1}(x){\phi_{n+1}^*(y)}| \leq 2 e^{8C} \, ,
\end{equation}
where in the last step we apply Lemma \ref{lem:poly_bd_gen} with index $n+1$ and constant $2C$.
The first term on the left of \eqref{darbouxxy},
multiplied by $e^{3C}$, we estimate from below with \eqref{cd3}
\begin{equation*}
 e^{3C}\left |y^n(x-y) \right | \ge e^C(1+2Cn^{-1})^n(1-Cn^{-1})^n ||y-s|-|x-s||\, . 
\end{equation*}
Noting that $2|y-s|\le |x-s|$ by \eqref{cd3}, we have the lower bound 
\begin{equation} \label{eq:CD_denom_estimate}
  e^{3C}\left |y^n(x-y) \right | \ge (1+Cn^{-1}-2(Cn^{-1})^2)^n  \left |s-x \right |
 \ge \left |s-x \right |\, ,\end{equation}
 where in the last step we used $n\ge 2C$.
The estimates \eqref{darbouxxy} and \eqref{eq:CD_denom_estimate} show \eqref{kernel estiamte K 2} and complete the proof of the lemma.




\end{proof}

We now prove Theorem \ref{fejerconvergence}. Assume $s\in E(\mu)$. Let $n \geq 2C$ and $z, \lambda \in B(s, \frac{C}{n})$. 
Thanks to \eqref{mumubar}
and \eqref{switchstar}, the 
polynomials $\phi_n$ and $\ti \phi_n$ switch roles of left and right orthogonal polynomials when passing from
$\mu$ to $\overline{\mu}$. With the definition \eqref{reprokernel} of the reproducing kernel we thus obtain
for the  kernel $\ti K_n$ associated to $\overline{\mu}$:
\begin{equation}
    \ti K_n (z, \lambda) = {K_n^*(\lambda, z)}\, .
\end{equation}
The reproducing  property \eqref{reproducing}, first for $\ti K_n$ and then $D_n$, yields 
\begin{equation}\label{cd10}
   {D_n (z, \lambda)} = { \int\limits_{\T}  {{D}_n (y, \lambda) }  \ti K_n ^* (y, z) d\overline{\mu} (y)} 
\end{equation}
\[
= \int\limits_{\T} K_n (z, y)  D_n (y,\lambda) \left ( d  \overline{ \mu} (y) - \overline{w(s)} \meas{y} \right )  +  \overline{w(s)} {\int\limits_{\T} {K_n (z,y)   D_n (y, \lambda)} \meas{y} }
\]
\[
= \int\limits_{\T} K_n (z,y)  {D}_n ( y, \lambda) \left ( d  \overline{\mu} (y) - \overline{w(s)} \meas{y} \right )  + \overline{w(s)} {K_{n} (z, \lambda)} \, .
\]
Rewriting the left side of \eqref{eq:KLbound}
with \eqref{cd10} and then estimating with Lemma \ref{lem:kbounds}, which analogously holds for $D_n$, gives\begin{equation}\label{firstthm1}
    \frac{1}{n+1}| \overline{w(s)} K_n (z, \lambda) - D_n (z, \lambda)|
    \end{equation}
    \begin{equation*}\leq  \frac{1}{n+1}\int\limits_{\T} |K_n (z,y)|  |D_n (y,\lambda)| \left | d \mu (y) - w(s) \meas{y} \right | 
\end{equation*}
\[
\leq e^{24C} \int\limits_{\T} \min \left \{ n+1, \frac{1}{(n+1)|y-s|^2} \right \}  \left | d \mu (y) - w(s) \meas{y} \right |  \, .
\]
This completes the proof of Theorem \ref{fejerconvergence}.

\section{Local parameters, first estimates}
\label{sec:lpfe}

Let $\mu\in \mathcal{T}_-$.
We prove some estimation and  approximation lemmas for local parameters $A_{n,s},B_{n,s},  \ti A_{n,s},\ti B_{n,s}$ 
of the normalized one sided orthogonal polynomials. These local parameters are  
sums and differences over shifts by a $2n$-th root of unity: define $A_{n,s}$ and $B_{n,s}$ as in \eqref{def:A_B_n}, and for $\gamma_n$ as in \eqref{gamman}, define the analogous
    \begin{equation}\label{def:ti A_B_n}
        \ti A_{n,s} := \frac{\ti \phi_n(s)-\ti \phi_n(s\gamma_n)}{2s^n} \, , \qquad 
    \ti B_{n,s} := \frac{\ti \phi_n(s)+\ti \phi_n(s\gamma_n)}{2} \, . 
    \end{equation}
Let $s\in E(\mu)$. As $s$ will be fixed throughout the section, 
we abbreviate notation by  omitting the subscript $s$ in
 \eqref{def:A_B_n} and 
 \eqref{def:ti A_B_n}.
Note that Lemma \ref{lem:poly_bd_gen} implies the bound \begin{equation}\label{eq:bd_A_B}
     |A_n|, |B_n |, |\ti A_n|, |\ti B_n |\leq \sqrt{2}   \, .  
    \end{equation}
 
The local parameters $A_n, B_n$ provide a local approximation of $\phi_n$ around $s$. 
\begin{lemma}[Local approximation of polynomials]\label{approximation with A B C D}
Let $C\geq 4$ and $4C\le n$ and $|z-s|\le C/n$. Then
    \begin{equation}\label{eq lemma abcd}
        \left| \phi_n(z) - A_nz^n-B_n \right|  \le   e^{36C} L(\mu,s,n)\, . 
    \end{equation}
\end{lemma}
\begin{proof}
Let $C$, $n$, and $z$ be given as in the lemma. Define $y:=\overline{z^{-1}}$. Note that 
\begin{equation}\label{upperlowerz}
  \frac 34\le |z|, |y|\le  \frac 43 \ .  
\end{equation}
Using the lower bound on $|z|$ and the fact that $|s|=1$, we estimate
\begin{equation}\label{ysestimate}
    |1-ys^{-1}|=|y-s| = |\overline{z^{-1}} - \overline{s^{-1}} | = \frac{|\overline{s}-\overline{z}|}{|\overline{sz}|} \le 2C/n \, ,
\end{equation} which also implies
\begin{equation}\label{ysgammaestimate}
    |1-y(s\gamma_n)^{-1}| \leq |y||1-\gamma_n| + |1-ys^{-1}| \leq \frac{4}{3} \pi/n + 2 C/n \leq 4 C/n \, ,
\end{equation}
By Theorem \ref{fejerconvergence}, each of the three expressions
\begin{equation}\label{eq:CD_conv_zs_2}
-\overline{ w  (s)}K_{n-1} (y, s) +  D_{n-1} (y,s)\, ,
\end{equation}
\begin{equation} \label{eq:CD_conv_zgamma_2}
    -\overline{ w  (s)}K_{n-1} (y,s \gamma_n) + D_{n-1} (y, s\gamma_n)\, ,
    \end{equation}
\begin{equation} \label{eq:CD_conv_sgamma_2}
    -\overline{ w  (s)}K_{n-1} (s, s \gamma_n)+ D_{n-1} (s, s \gamma_n) \end{equation}
is bounded in absolute value by 
\[
 n e^{31C} L(\mu,s,n) \, .
\]
By the Christoffel-Darboux formula of Lemma \ref{christoffeldarboux} applied to $K_n$ and $D_n$, and the bounds \eqref{ysestimate} and \eqref{ysgammaestimate} to estimate the left-most terms in \eqref{eq:CD},
each of the three expressions 
\begin{equation}\label{eq:CD_conv_zs}
\overline{ w  (s)}[\ti \phi_n (y) \phi_n ^* (s) - y^n s^{-n} \phi_n^*(y) \ti\phi_n (s)] +  (1-y^n s^{-n} )
\end{equation}
\begin{equation} \label{eq:CD_conv_zgamma}
    \overline{ w  (s)}[\ti \phi_n (y) \phi_n ^* (s\gamma_n) + y^n s^{-n} \phi_n^*(y) \ti\phi_n (s\gamma_n)] + (1+y^n s^{-n} )
    \end{equation}
\begin{equation} \label{eq:CD_conv_sgamma}
    \overline{ w  (s)}[ \ti \phi_n (s) \phi_n ^* (s\gamma_n) +  \phi_n^*(s) \ti\phi_n (s\gamma_n)] + 2 \end{equation}
is bounded in absolute value by 
\[
4C e^{31 C} L(\mu,s,n) \leq e^{32C} L(\mu,s,n) \, .
\]
We multiply \eqref{eq:CD_conv_zs} by  $\phi_n ^* (s\gamma_n )$, \eqref{eq:CD_conv_zgamma} by $-\phi_n ^* (s)$
and \eqref{eq:CD_conv_sgamma} by $s ^{-n} y^n \phi_n ^* (y)$,  and then sum up the three terms.
Note that all terms arising in square brackets cancel: the first summand in brackets of \eqref{eq:CD_conv_zs} cancels
the first summand in brackets of \eqref{eq:CD_conv_zgamma}, the second summand in brackets in \eqref{eq:CD_conv_zs} cancels the 
first summand in \eqref{eq:CD_conv_sgamma}, while 
the second summand in brackets in \eqref{eq:CD_conv_zgamma} cancels the
second summand in  \eqref{eq:CD_conv_sgamma}.

Using Lemma \ref{lem:poly_bd_gen} and the triangle inequality, we obtain
\[ |(1-y^n s^{-n} )\phi_n ^* (s\gamma_n )-(1+y^n s^{-n})\phi_n ^* (s)+2s ^{-n} y^n \phi_n ^* (y)|
\leq  e^{35C} L(\mu,s,n) \, .
\]
Rewriting $y$ in terms of $z$ and conjugating inside the absolute value signs gives
\[ |(1-z^{-n} s^{n} )\phi_n (s\gamma_n )-(1+z^{-n} s^{n})\phi_n  (s)+2s ^{n} z^{-n} \phi_n (z)|
\leq e^{35C} L(\mu,s,n)\, .
\]
Finally, we multiply both sides by $\frac{s^{-n} z^n}{2}$, which is bounded by $e^{C}$, to obtain
\[ |  2^{-1} (s ^{-n} z^{n}-1 )\phi_n (s\gamma_n )- 2^{-1} (s ^{-n} z^{n}+1)\phi_n  (s)+ \phi_n (z)|
\leq e^{36C} L(\mu,s,n) \, .
\]
This is \eqref{eq lemma abcd} by definition of $A_n$ and $B_n$, which proves Lemma \ref{approximation with A B C D}.
\end{proof}

\begin{lemma}\label{lem:liminf_A_B_nonzero}
    Let $n\ge 40$.  We have
\begin{equation}\label{sum squares ab}
|A_n|^2+|B_n|^2+|\ti A_n|^2+|\ti B_n|^2 =2 \, .
\end{equation}

\begin{equation}\label{product ab}
|A_n  \overline{B_n} + \ti A_n \overline{\ti B_n}| \leq e^{372} L(\mu, s, n)    
\end{equation}

\begin{equation}\label{difference ab square}
 \left | -\ti A_n \overline{A_n} +    \ti B_n \overline{B_n} + \frac{1}{\overline{w (s)}} \right |
\leq  e^{370} L (\mu, s, n) \, .   
\end{equation}    
\end{lemma}
\

\begin{proof}
Note \eqref{sum squares ab} follows from the determinant identity \eqref{eq:ortho_polys_det2}, since we compute
\[
2(|A_n|^2 + |B_n|^2 + |\ti A_n|^2 + |\ti B_n|^2) = |\phi_n (s)|^2 + |\ti \phi_n(s)|^2 +  |\phi_n (s \gamma_n)|^2 + |\ti \phi_n(s \gamma_n )|^2 = 4 \, .
\]

Set $C=10$, and note $n \geq 4C$. 
If for some complex numbers $X,Y$, we have  $|X-Y|  < M$, then 
\[
| |X|^2 - |Y|^2 |  = ||X| - |Y| | (|X| + |Y|) \leq M (|X| + |Y|)\, .
\]
This with the local approximation Lemma \ref{approximation with A B C D} and the  bounds \eqref{eq:bd_A_B} gives  for $t\in B(s,C/n) \cap \T$
\[
| |A_n t^n + B_n|^2 - |\phi_n|^2 | \leq e^{365} L (\mu, s, n) \, .
\]
Using the analogous inequality for $\ti \phi_n$,  the triangle inequality, and  the identity  \eqref{eq:ortho_polys_det2} for the sum of squares of $\phi_n$, $\ti \phi_n$, gives
\begin{equation}\label{eq:approx_det_id}
| |A_n t^n + B_n|^2 + |\ti A_n t^n + \ti B_n|^2 -2| \leq e^{370} L(\mu, s, n) \, .
\end{equation}
Using \eqref{sum squares ab}, the left side of \eqref{eq:approx_det_id} can be written as
\begin{equation}\label{two real part}
    \left|2 \Re (t^n (
  A_{n} \overline{ B_{n}} + \ti{A_{n}} \overline{\ti B_{n}} )) \right|\ .
\end{equation} 
 Inequality \eqref{product ab} follows by picking $t$ so that 
 \[
   \Re (t^n (
  A_{n} \overline{ B_{n}} + \ti{A_{n}} \overline{\ti B_{n}} )) = A_{n} \overline{B_{n}} + \ti{A_{n}} \overline{\ti B_{n}} \, ,
\]
which is possible as $t^n$ takes all values in $\T$ on any arc of $\T$ of length 
$2\pi/n$. 
Taking a limit in \eqref{eq:CD_conv_sgamma} while using upper bounds of Lemma \ref{lem:poly_bd_gen}
on $|\phi_n|$ $|\ti \phi_n|$  by $\sqrt{2}$ on $\T$
shows $|w(s)|>1/4$.
Thus dividing   \eqref{eq:CD_conv_sgamma} by $\overline{w}(s)$, we obtain
\[
| \ti \phi_n (s) \phi_n ^* (s\gamma_n) +  \phi_n^*(s) \ti\phi_n (s\gamma_n) + \frac 2{\overline{ w (s)}}| \leq e^{325} L(\mu, s, n) \, .
\]
Using the upper bounds on $\phi_n$ and $\ti \phi_n$
again and 
 the definition of $A_n$ and $B_n$
\[
| (\ti A_n s^n + \ti B_n)  (  -\overline{A_n}s^{-n} + \overline {B_n} ) +  (\overline{A_n}s ^{-n} + \overline{B_n} ) ( -\ti A_n s^{n} + \ti B_n) + \frac{2}{\overline{w(s)}}| 
\leq e^{370} L(\mu, s, n) \, .
\]
As the cross terms on the left side vanish, dividing by $2$ we obtain \eqref{difference ab square}.
\end{proof}

\begin{lemma}\label{lem:lwr_bd_w}
 We have $|w(s)|\ge 1\, .$
\end{lemma}
\begin{proof}
We have for $n\ge 40$,
using firstly \eqref{difference ab square} and secondly \eqref{sum squares ab}:
\begin{equation}\label{betterw1}
 |w(s)|^{-1}\le |A_n \ti A_n|+ |B_n \ti B_n| + e^{370}L(\mu, s,n) 
\end{equation}
\begin{equation*}\label{betterw2}
= 1- \frac 12(|A_n|- |\ti A_n|)^2- \frac 12 (|B_n|- |\ti B_n|)^2 +e^{370}L(\mu, s,n) 
\le 1+e^{370}L(\mu, s,n)\ \, .
\end{equation*}
Taking a limit as $n$ tends to $\infty$ proves the lemma.
\end{proof}

The following two lemmas say that only the $A$'s or the $B$'s can be small, and whichever is small determines whether $ \phi_n ^* \ti \phi_n $ is close to $-\overline{w(s)^{-1}}$ or $\overline{w(s)^{-1}}$.

\begin{lemma}\label{abtocdlem}

Let $0<\eta \le \frac 12$.
Let $n_0\ge 40$ be large enough so that for $n\ge n_0$ we have 
    \begin{equation}\label{abtocd1}
        e^{1000}L(\mu,s,n)\le \eta |w(s)|^{-2} \, .
    \end{equation}
Let $n\ge n_0$.
Let $(C,D,\ti C, \ti D)$ be one of the four tuples
\begin{equation*}
 (|A_n|,|B_n|,|\ti A_n|,|\ti B_n|)\, , \ (|\ti A_n|,|\ti B_n|, |A_n|, |B_n|)\, , 
\end{equation*}
\begin{equation*}
 (|B_n|,|A_n|,|\ti B_n|,|\ti A_n|)\, , \ (|\ti B_n|,|\ti A_n|, |B_n|, |A_n|)\, .
\end{equation*}
If
\begin{equation}\label{abtocdhyp}
 \eta \le C\le (8|w(s)|)^{-1}   \, ,
\end{equation}
then
\begin{equation}\label{abtocdcon1}
 (4|w(s)|)^{-1}  \le \ti D  \le 2\, ,
\end{equation}
\begin{equation}\label{abtocdcon2}
  (4|w(s)|)^{-1} \le  D  \le 2\, ,
\end{equation}
\begin{equation}\label{abtocdcon3}
(9|w(s)|)^{-1}{C} \le  \ti C\le (9|w(s)|){C}   \, .
\end{equation}

\end{lemma}
\begin{proof}
Throughout the proof, we shall freely use the upper bounds on $C,D,\ti C, \ti D$ by $2$ from \eqref{eq:bd_A_B}.
The triangle inequality with \eqref{difference ab square}, the assumed upper bound in \eqref{abtocdhyp} and the lower bound on $|w|$ in Lemma \ref{lem:lwr_bd_w} gives
 \begin{equation}\label{abtocd2}
        (2|w(s)|)^{-1}\le D\ti D \, .
    \end{equation}
This implies \eqref{abtocdcon1} and \eqref{abtocdcon2}.
By \eqref{product ab} and the lower bounds  \eqref{abtocdhyp} and \eqref{abtocdcon2}, we obtain    
\begin{equation}\label{abtocd4}\left|\frac {CD}{\ti D}-\ti C \right| \le  e^{-100} C |w(s)|^{-1}\, .
    \end{equation}
Using \eqref{abtocdcon1}, \eqref{abtocdcon2} and the lower bound on $|w|$ from Lemma \ref{lem:lwr_bd_w}, we obtain \eqref{abtocdcon3}.
\end{proof}

\begin{lemma}\label{smallablem}
    Let $0<\eta<10^{-10} |w(s)|^{-2}$.
    Let $n_0\ge 40$ be large enough so that for $n\ge n_0$ we have 
\begin{equation}\label{smallabhyp1}
e^{1000}L(\mu,s,n)\le \eta^2 \, .
    \end{equation}
Let $n\ge n_0$. If $|A_n|<\eta^2$ or $|\ti A_n| <\eta^2$, then
\begin{equation}\label{smallabcon1}
\left| \phi_{n} ^* (s)\ti \phi_{n} (s)  +
\overline{w(s)^{-1}}\right|\le \eta \, .
\end{equation}
If $|B_n|<\eta^2$ or $|\ti B_n| <\eta^2$, then
\begin{equation}\label{smallabcon2}
\left|\phi_{n} ^* (s)\ti \phi_{n} (s) -
\overline{w(s)^{-1}}\right|\le \eta \, .
\end{equation}
\end{lemma}

\begin{proof}

 Assume first $|A_{n}|\le \eta^2$.
 Lemma \ref{abtocdlem}, and then the upper bound on $|w|^{-1}$ from Lemma \ref{lem:lwr_bd_w}, imply 
\[
|\ti A_{n}| \leq 9|w(s)| |A_n| \leq 10 |w(s)| \eta^2 \leq 10^{-8}\eta \, .
\]
Then \eqref{difference ab square}  implies
\begin{equation}\label{subs101}
\left|\ti B_{n} \overline{B_{n}}+
\overline{w(s)^{-1}}\right|\le 10^{-4}\eta \, .
\end{equation}
With the local approximation Lemma \ref{approximation with A B C D}, we obtain \eqref{smallabcon1}.
If instead $|\ti A_{n}|\le \eta^2$,
 Lemma \ref{abtocdlem} 
implies $| A_{n}|\le 10^{-8}\eta$ and we obtain 
again \eqref{subs101} and 
\eqref{smallabcon1}.
Assume now $|B_{n}|\le \eta^2$. Lemma \ref{abtocdlem} 
implies $|\ti B_{n}|\le 10^{-8}\eta$ and \eqref{difference ab square} implies
\begin{equation}\label{subs210}
\left|\ti A_{n} \overline{A_{n}}-
\overline{w(s)^{-1}}\right|\le 10^{-4}\eta \, .
\end{equation}
Now the local approximation Lemma \ref{approximation with A B C D} gives \eqref{smallabcon2}. Similarly,
if  $|\ti B_{n}|\le \eta^2$,
we conclude $| B_{n}|\le 10^{-8}\eta$ and obtain 
\eqref{subs203} and \eqref{smallabcon2}. This completes the proof of the lemma.

\end{proof}

\section{Local parameters, refined estimates}

\label{sec:lpre}

The estimates in this section will be used in the second
part of Theorem \ref{thm:conv_to_AB}.
We continue to fix $\mu\in \mathcal{T}_-$ and $s\in E(\mu)$. The next two lemmas consider the case the left side of \eqref{product ab} is small, i.e., in what follows one should take $(a,b,c,d) = (A_n, B_n, \ti A_n, \ti B_n)$ and $\epsilon = e^{372} L(\mu,s,n)$.

\begin{lemma}\label{abcdlem}
    Let $\epsilon,\eta>0$. Let $a,b,c,d$ be complex numbers 
    bounded below in absolute value by $\eta$ and above in absolute value by $2$.
Assume
\begin{equation}\label{abcdhyp}
    |a\overline{b}+c\overline{d}|\le \epsilon\, .
\end{equation}
Then 
\begin{equation}\label{abcdcon}
    \left||\overline{a}c-\overline{b}d|- (|ac|+|bd|)\right|\le
    4\eta^{-1} \epsilon\, .
\end{equation}
\end{lemma}

\begin{proof}
We multiply the left side of \eqref{abcdcon} by $|a|$
and then apply the triangle inequality with \eqref{abcdhyp} 
twice to obtain
\begin{equation*}
    \left||a\overline{a}c-a\overline{b}d|- (|a^2c|+|abd|)\right|
\end{equation*}
\begin{equation}\label{abcd01}
   \le  \left||a\overline{a}c+c\overline{d}d|- (|a^2c|+|cd^2|)\right|+ 2|d|\epsilon =2|d|\epsilon\, .
\end{equation}
Using $|d|\le 2$ and $|a|\ge \eta$ proves \eqref{abcdcon}
and completes the proof of the lemma.
\end{proof}

\begin{lemma}\label{4abcdlem}
    Let $\epsilon >0$ and $a,b,c,d$ be reals 
    bounded above by $\sqrt{2}$. If both 
\begin{equation}\label{4abcdhyp1}
        a^2+b^2+c^2+d^2=2
    \end{equation}
    and
\begin{equation}\label{4abcdhyp2}
    |ab-cd|\le \epsilon\, ,
\end{equation}
then
\begin{equation}\label{4abcdcon}
    \left|(ac+bd)^2(ad+bc)^2- 4abcd\right|\le 1000 \epsilon\, .
\end{equation}
\end{lemma}
\begin{proof}
From \eqref{4abcdhyp1}, it suffices to prove 
\begin{equation}\label{4abcd1}
    |(ac+bd)^2(ad+bc)^2- (a^2+b^2+c^2+d^2)^2abcd|\le 1000\epsilon\, .
\end{equation}
Expanding terms, sorting the binomial expressions to
account first for pure squares and then mixed terms gives
for the polynomial inside the absolute values
\begin{equation}\label{4abcd2}
     a^4 c^2 d^2+ a^2 b^2 c^4 + a^2 b^2 d^4+ b^4 c^2 d^2
\end{equation}
\begin{equation}\label{4abcd3}
    + (a^2c^2  +  b^2d^2 +  a^2d^2 +  b^2 c^2) (2abcd)
\end{equation}
\begin{equation}\label{4abcd4}
    + 4(abcd)^2
  \end{equation}
\begin{equation}\label{4abcd5}
    - (a^4 + b^4 + c^4 + d^4)(abcd)
\end{equation}
\begin{equation}\label{4abcd6}
    - (a^2b^2 + a^2c^2+a^2d^2+b^2c^2+b^2d^2+c^2d^2)(2abcd) .
\end{equation}
Thanks to hypothesis \eqref{4abcdhyp2} and the upper bound on $a,b,c,d$, 
we may replace $abcd$ in each individual summand by $(ab)^2$ or $(cd)^2$ as 
we desire, up to an error of at most $64\epsilon$.

Therefore, up to such error, the four terms in \eqref{4abcd2}
match up with the four terms in \eqref{4abcd5}. Thus these two lines differ by at most $256\epsilon$.

Furthermore, four of the six terms in \eqref{4abcd6} match up
exactly with the four terms in \eqref{4abcd3} and thus cancel.
The remaining two terms in \eqref{4abcd6}, again up to an error of 
$64\epsilon$, each match with half of \eqref{4abcd4}.
This proves \eqref{4abcd1} and thus \eqref{4abcdcon},
completing the proof of the lemma.

\end{proof}
   
\begin{lemma}\label{approx with zeros}
   Let $0<\eta< 2^{-1}$,
   and $n\geq 400$. Assume $\eta \le |A_n|,|B_n|,|\ti{A}_n|,|\ti{B}_n|$. Then
\begin{equation}\label{eq:square_prod_An_zeros1}
\left |
(|\ti A_n A_n| +|\ti{B}_n B_n|)^2 - \frac{1}{|w|^2} \right |\leq \eta^{-1} e^{385} L(\mu,s,n)\, ,
\end{equation}

\begin{equation}\label{eq:solve_r_w_poly1}
\left| 4 \left(\left| \frac{A_n \ti{B}_n}{B_n \ti{A}_n}\right|^{\frac 12}+ \left|\frac{B_n \ti{A}_n}{A_n \ti{B}_n}\right|^{\frac 12}\right)^{-2}  - \frac{1}{|w|^2} \right | \leq \eta^{-4} e^{400} L(\mu,s,n) \, .
\end{equation}

\end{lemma}
\begin{proof}

We first prove \eqref{eq:square_prod_An_zeros1}.
Applying Lemma \ref{abcdlem} to \eqref{product ab}  gives
\begin{equation}\label{abcdappl}
\left |
|\overline{A_n} \ti{A}_n- \overline{B}_n \ti {B}_n|
- (|A_n \ti{A}_n|+|B_n\ti{B}_n| ) \right|\le \eta^{-1}
e^{375} L(\mu, s, n)    .
\end{equation}
Applying further 
\eqref{difference ab square}
with the triangle inequality gives with $\eta<2^{-1}$
\begin{equation}\label{abcdappl2}
\left |
|A_n \ti{A}_n|+|B_n\ti{B}_n| -\frac 1{|w(s)|}\right|\le \eta^{-1}
e^{380} L(\mu, s n)  \, .
\end{equation}
With $|w(s)|\ge 1$ from Lemma \ref{lem:lwr_bd_w} and the upper bound on $A_n, B_n, \ti A_n, \ti B_n$ from \eqref{eq:bd_A_B}, multiplying Inequality \eqref{abcdappl2} by
\begin{equation*}
|A_n \ti{A}_n|+|B_n\ti{B}_n| +\frac 1{|w(s)|}\le 5  \, ,
\end{equation*}
we obtain \eqref{eq:square_prod_An_zeros1}, thus completing this part of the proof.

Next, we prove \eqref{eq:solve_r_w_poly1}.
By \eqref{eq:square_prod_An_zeros1} and the triangle inequality,
it suffices to prove
\begin{equation}\label{eq:square_ratios}
\left |
(|\ti A_n A_n| +|\ti{B}_n B_n|)^2 - 4 
\left(\left| \frac{A_n \ti{B}_n}{B_n \ti{A}_n}\right|^{\frac 12}+ \left|\frac{B_n \ti{A}_n}{A_n \ti{B}_n}\right|^{\frac 12}\right)^{-2}
\right |\leq  \eta^{-4} e^{395} L(\mu,s,n)\, .
\end{equation}
Using
\begin{equation}\label{auxaabb}
    \eta^2\le |B_n\ti{A}_n|+|A_n \ti{B}_n| \, ,
\end{equation}
it suffices to prove 
\begin{equation}\label{eq:square_ratios2}
\left |
(|\ti A_n A_n| +|\ti{B}_n B_n|)^2
(|B_n\ti{A}_n|+|A_n \ti{B}_n|)^2- 4 \left|{A_n \ti{B}_n}{B_n \ti{A}_n}\right|\right|\leq e^{395} L(\mu,s,n)\, .
\end{equation}
This however follows from applying Lemma \ref{4abcdlem} 
to $(a,b,c,d) = (|A_n|, |B_n|, |\ti A_n|, |\ti B_n|)$,
 using \eqref{sum squares ab} and
\eqref{product ab} 
to verify the assumptions of Lemma \ref{4abcdlem}.
 This completes the proof of  \eqref{eq:solve_r_w_poly1}
and thus the proof of Lemma \ref{approx with zeros}.

\end{proof}

The following lemma basically says in \eqref{abybcon} that $|A_n \ti B_n|$ and $|\ti A_n B_n|$ are roughly comparable.
\begin{lemma}\label{abyblem}
Let $0 < \eta < 10^{-10}|w(s)|^{-2}$ and let $n_0>400$ be large enough so that for $n\ge n_0$ we have 
    \begin{equation}\label{abybhyp1}
        e^{1000}L(\mu,s,n)\le \eta ^8 \, .
    \end{equation}
Let $n>n_0$. Let $(C,D,\ti C,\ti D,\sigma)$ be one of the tuples
\[({A_n},{B_n}, {\ti A_n}, \ti B_n,1)\, , \ 
(\overline{\ti A_n},\overline{\ti B_n}, \overline{A_n}, \overline{B_n},1)\, , \]
\[
({B_n},{A_n}, {\ti B_n}, \ti A_n,-1)\, , \ (\overline{\ti B_n},\overline{\ti A_n}, \overline{B_n}, \overline{A_n},-1)\, . 
\]
Assume both 
\begin{equation}\label{abybhyp2}
    \eta\le |C|\le 10^{-10}|w(s)|^{-2}\, ,
\end{equation}
\begin{equation}\label{abybhyp3}
\left|  \phi_n^*(s)\ti \phi_n(s)+\sigma \overline{w(s)^{-1}}\right|\le  \eta^5\, .
\end{equation}
Then
\begin{equation}\label{abybcon}
    \frac 1 4  \le \left|\frac DC \right| \left|1-\left|\frac{C\ti D}{\ti CD}\right|\right|\le 4 \, .
\end{equation}
\end{lemma}
\begin{proof}
Throughout the proof, we shall freely use the upper bounds  \eqref{eq:bd_A_B}
as well as the upper bounds of $|\phi_n(s)|$ and $|\ti\phi_n(s)|$  by $2$ 
by \eqref{eq:ortho_polys_det2}. 

Assumption \eqref{abybhyp3} and the local approximation Lemma \ref{approximation with A B C D} applied with constant $4$ give
\begin{equation}\label{abyb1}
\left|   \overline{A_{n}} \ti A_{n} +\overline{B_{n}} \ti B_{n} +  \overline{A_{n}}  \ti B_{n} s^{- n} + \ti A_{n} \overline{B_{n}} s^{n} +\sigma \overline{w(s)^{-1}}\right|\le  10\eta^5 \, .    
\end{equation}
This can be rewritten with suitable  $\tau\in \{-1,1\}$ as 
\begin{equation}\label{abyb2}
\left|   \overline{C} {\ti C} +\overline{D} \ti D +  \overline{C}  \ti D s^{-\tau n} + \ti C \overline{D} s^{\tau n} +\sigma \overline{w(s)^{-1}}\right|\le  10\eta^5 \, .    
\end{equation}
Using \eqref{difference ab square}, which may be rewritten as
\begin{equation} \label{difference ab square CD}
\left | -\ti C \overline{C} + \ti D \overline{D} + \sigma \overline{w^{-1}} \right | \leq e^{370} L(\mu,s,n)
\end{equation}
we obtain from \eqref{abyb2} that
\begin{equation}\label{abyb3}
\left|   2\overline{C} {\ti C}  +  \overline{C}  \ti D s^{\tau n} + \ti C \overline{D} s^{-\tau n} \right|\le  20\eta^5 \, .    
\end{equation}
Dividing by $\overline{C} \ti C$ and using Assumption \eqref{abybhyp2} and Lemma \ref{abtocdlem}, we obtain
\begin{equation}\label{abyb4}
 \left| 2 + \frac{{\ti D}}{\ti C } s^{\tau n} + \frac{\overline{D}}{ \overline{C} }s^{ -\tau n} \right|\le  10^{-5} \, .
\end{equation}
By \eqref{product ab}, which may be written as
\[
\left |C \ti C + D \ti D \right | \leq e^{372} L(\mu, s,n) \, ,
\]
and Lemma \ref{abtocdlem}, we have
\begin{equation}\label{abyb5}
\left|\frac{\ti C \overline{\ti D}}{C \overline{D}}
+1\right|,
\left|\left|\frac{\ti C \overline{\ti D}}{C \overline{D}}\right|
-1\right|
\le \eta^5 \, .
\end{equation}
Hence, applying Lemma \ref{abtocdlem} with \eqref{abyb5}, 
\begin{equation}\label{abyb6}
\left|\left|\frac{C \ti D}{\ti C D}\right|+\frac {C \ti D}{\ti C D}\right|\le 
\left|\left|\frac{C \ti D}{\ti C D}
\frac{\ti C \overline{\ti D}}{C \overline{D}}\right|-\frac{C \ti D}{\ti C D}
\frac{\ti C \overline{\ti D}}{C \overline{D}}
\right|+\eta^2= \eta^2  \, .
\end{equation}
Using \eqref{abyb6} in \eqref{abyb4} gives
\begin{equation}\label{abyb7}
 \left| 2 - \left|\frac{C \ti D}{\ti C D}\right|\frac{{ D}}{ C } s^{ \tau n} + \frac{\overline{D}}{ \overline{C} }s^{ -\tau n} \right|\le  10^{-4} \, .
\end{equation}
Separating real and imaginary part in \eqref{abyb7} gives
\begin{equation}\label{abyb8}
 \left| 2 - \left(\left|\frac{C \ti D}{\ti C D}\right|-1\right)
 \Re \left(\frac{{ D}}{ C } s^{ \tau n}\right) \right|\le  10^{-4} \, ,
\end{equation}
\begin{equation}\label{abyb9}
 \left| \left(\left|\frac{C \ti D}{\ti C D}\right|+1\right)
 \Im \left(\frac{{ D}}{ C } s^{ \tau n}\right) \right|\le  10^{-4} \, .
\end{equation}
Inequalities \eqref{abyb8} and \eqref{abyb9} together with \eqref{abyb5} show that the imaginary part of $(D/C) s^{\tau n}$ is less than $10^{-2}$ times the real part: indeed, using first \eqref{abyb8}, then \eqref{abyb9} and finally \eqref{abyb5}, we compute
\[
10^{4} \left | \Im \left(\frac{{ D}}{ C } s^{ \tau n}\right) \right  | \leq 1  \leq  (2 - 10^{-4}) \leq  \left | \left|\frac{C \ti D}{\ti C D}\right|-1\right |
 \left | \Re \left(\frac{{ D}}{ C } s^{ \tau n}\right) \right | \leq  \eta^{5} \left | \Re \left(\frac{{ D}}{ C } s^{ \tau n}\right) \right | \, .
\] 
Hence the ratio between real part and absolute value of $(D/C)s^{\tau n}$ is between $9/10$ and $10/9$.
Thus by \eqref{abyb5} and \eqref{abybhyp1}, in \eqref{abyb8} we can replace the real part of $(D/C) s^{\tau n}$ by its absolute value to get
\begin{equation}\label{abyb10}
 \left| 2 - \left | \left|\frac{C \ti D}{\ti C D}\right|-1\right |
 \left | \frac{{ D}}{ C } \right| \right|\le  10^{-3} \, ,
\end{equation} which  implies \eqref{abybcon} and completes the proof of the lemma.

\end{proof}

\begin{lemma}\label{abclem}

Let $0<\eta \leq 10^{-10} |w(s)|^{-2}$ and let $n_0>400$ be large enough so that for $n\ge n_0$ we have \begin{equation}\label{abchyp1}
e^{1000}L(\mu,s,n)\le \eta ^8 \, .
    \end{equation}
Let $n>n_0$. 
Let $(C,\sigma)$ be one of the tuples $(A_n,1)$, $(B_n,-1)$, $(\overline{\ti A_n}, 1)$, $(\overline{\ti B_n},-1)$. 
If $|w(s)|>1$, assume
\begin{equation}\label{abchyp2}
 \eta \le |C|\le 10^{-10}|w(s)|^{-2}\min(1,|w(s)|-1)   \, .
\end{equation}
If $|w(s)|=1$, assume
\begin{equation}\label{abchyp3}
 \eta \le |C|\le 10^{-10}\, .
\end{equation}
Then 
\begin{equation}
    \label{abccon}
    \eta^5 \le \left|\phi_n^*\ti \phi_n+ \sigma {\overline{w(s)^{-1}}}\right|\le (2|w(s)|)^{-1}
  \, .
\end{equation}
\end{lemma}

\begin{proof}

Given $(C,\sigma)$
as in the lemma, we consider the unique completion to a tuple
$(C,D,\ti C, \ti D,\sigma)$ as listed in Lemma \ref{abyblem}.
By  \eqref{difference ab square}, or equivalently \eqref{difference ab square CD}, as well as \eqref{abchyp1}, \eqref{abchyp2}, and \eqref{abchyp3},  we have
\begin{equation}\label{abc3}
\left| \ti D\overline{D}+\sigma \overline{w(s)^{-1}}\right|\le (20|w(s)|)^{-1}\, .
\end{equation}
By the local approximation Lemma \ref{approximation with A B C D} applied with constant $4$ and the Lemma \ref{abtocdlem}, we obtain from \eqref{abc3} 
\begin{equation}\label{abc4}
\left|\ti \phi_n(s)  \phi_n^*(s)+\sigma \overline{w(s)^{-1}}\right|\le (10|w(s)|)^{-1} \, .
\end{equation}
This shows one inequality in  \eqref{abccon}. 
To prove  the other inequality, assume to get a contradiction that
\begin{equation}\label{abc6}
\left|  \phi_n^*(s)\ti \phi_n(s)+\sigma \overline{w(s)^{-1}}\right|\le  \eta^5\, .
\end{equation}
Set
\begin{equation}\label{abc10}
    t:= \left| {C \ti D }{ (\ti C D)^{-1}}\right|\,  
\end{equation}
and apply Lemma \ref{abyblem} to obtain
\begin{equation}\label{abc60}
    4^{-1}\le |DC^{-1}||1-t|\le 4\, .
\end{equation}
By \eqref{eq:solve_r_w_poly1}, we have
\begin{equation}\label{abc11}
|4( t+2+t^{-1})^{-1}  - |w(s)|^{-2}|\le \eta^6\, .\, 
\end{equation}
Note that as $\eta \to 0$, we have $\max(t,t^{-1})$ goes to $v:=\mathrm{exp}(2 \mathrm{arccosh}(|w|))$. In what follows, we now analyze three different cases one after the other: the case when $t$ is very far from $1$ and $v$, the case when $t$ is close to $1$ but nearer to $v$, and then the case where $t$ is very close to $1$. 
Assume first $|w(s)|>1.01$.
Then \eqref{abc11} gives $\max(t,t^{-1})>1.001$ and thus 
\begin{equation}\label{abc43}
|t-1|>10^{-4} \, .
\end{equation}
But then \eqref{abc60} gives $|DC^{-1}|\le 10^5$. This gives a contradiction
with 
\eqref{abchyp2}, and \eqref{abchyp1} and
Lemma \ref{abtocdlem}.
Assume next $1<|w(s)|\le 1.01$.
Then \eqref{abc11} gives $\max(t,t^{-1})<1.1$ and thus 
\begin{equation}\label{abc44}
|t-1|<0.2 \, .
\end{equation}
We obtain from \eqref{abc11}
\begin{equation}\label{abc45}
|4|w(s)|^2-t-2-t^{-1}|\le \eta^5 \, ,
\end{equation}
or equivalently
\begin{equation}\label{abc47}
|4(|w(s)|+1)(|w(s)|-1)-{(t-1)^2}t^{-1}|\le \eta^5 \, .
\end{equation}
It follows, using the bound on $\eta$ from \eqref{abchyp2}, that
\begin{equation}\label{abc48}
|t-1|\ge (|w(s)|-1)^\frac 12\ge |w(s)|-1 \, .
\end{equation}
Then \eqref{abc60} gives $|DC^{-1}|\le 10(|w(s)|-1)^{-1}$. This gives a contradiction
with 
\eqref{abchyp2}, and \eqref{abchyp1} and
Lemma \ref{abtocdlem}.
Finally, assume $|w(s)|=1$.
Then \eqref{abc11} implies
\begin{equation}\label{abc61}
|t-1|\le \eta^3
\end{equation}
Then \eqref{abc60} gives $|DC^{-1}|\ge 4^{-1} \eta^{-3}$.
This gives again a contradiction
with 
\eqref{abchyp3}, and \eqref{abchyp1} and
Lemma \ref{abtocdlem},
and completes the proof of the lemma.

\end{proof}

\section{Proof of Theorem \ref{thm:conv_to_AB}}

Once again, we freely use the upper bounds \eqref{eq:ortho_polys_det2} on the normalized orthogonal polynomials and \eqref{eq:bd_A_B} on the local parameters in this section.

We now prove that \eqref{subshyp1} implies \eqref{subscon1}.
Let $\mu\in \mathcal{T}_-$, let $s\in E(\mu)$, and let $(n_k)$ be a monotone increasing sequence of natural numbers. Assume that \eqref{subshyp1} holds, we show
\eqref{subscon1}.
Let $0<\epsilon<10^{-10} |w(s)|^{-2}$.
By Lebesgue differentiation and
assumption \eqref{subshyp1}, there is $k_0$ such that for all $k>k_0$
\begin{equation}
    \label{subs100}
    e^{1000}L(\mu, s,n)\, ,\  \min(|A_{n_k}|,|B_{n_k}|)\le \epsilon^6\, .
\end{equation}
Let $k>k_0$. By Lemma \ref{smallablem}, we obtain
\begin{equation}\label{subs102}
\min \left ( \left| \phi_{n_k}^* (s) \ti \phi_{n_k} (s)  +
\overline{w(s)^{-1}}\right|,
\left|\phi_{n_k}^* (s) \ti \phi_{n_k} (s) -
\overline{w(s)^{-1}}\right| \right )
\le \epsilon^2 \, .
\end{equation}
With the binomial formula and upper bound $|w^{-1}|\leq 1$ from Lemma \ref{lem:lwr_bd_w}, we obtain
\begin{equation}\label{subs103}
\left| ( \phi_{n_k} ^* (s) \ti\phi_{n_k} (s) )^2 -
\overline{w(s)^{-2}}\right|
\le \epsilon \, .
\end{equation}
As $\epsilon$ was arbitrarily small, we obtain \eqref{subscon1}.
This completes the proof of the first part of Theorem \ref{thm:conv_to_AB}.
Before turning to the second part, we prove some auxiliary lemmas. In what follows, we let $|E|$ of a set $E \subset \T$ denote its arclength. Note that
\[
|E| = \int\limits_{\T} \mathbf{1}_{E} \, d|s| = 2 \pi \int\limits_{\T} \mathbf{1}_E \, , 
\]
since the last integral denotes the mean value of $\mathbf{1}_E$ on $\T$.

\begin{lemma}\label{phiphi}
    Let $a,b,c,d,z,y,x,v$ be complex numbers bounded in absolute value by $2$. Let $J$ be an arc on $\T$.
    Let $m,n$ be  positive integers with  
\begin{equation}\label{phiphihyp}
        10^7|J|^{-1}\leq m,n,|m-n|,|2m-n|, |m-2n|\, .
    \end{equation}
   Then  
   \begin{equation}
       \label{phiphicon}
       \int\limits_J ||a s^m+ b|| c s^{-m}+  d|- |z s^n+y|| x s^{-n}+ v|| \,  d|s|\,  \ge 
10^{-2} \min(|ab|,|cd|,|zy|,|xv|)^4|J|\, . 
   \end{equation}
\end{lemma}
\begin{proof}
Define, for $s\in \T$,
\begin{equation}\label{phiphi2}
   o(s):= |a s^m+ b|^2| c s^{-m}+ d|^2- |z s^n+y|^2|x s^{-n}+v|^2
    \, .
\end{equation}
Expanding squares, we obtain  for some real number $k$
depending on $a$ through $v$ and with dots representing
analogous expressions with letters $z,y,x,v$
\begin{equation}\label{phiphi3}
   o(s)= (|a|^2+|b|^2+ a\overline{b}s^m+b\overline{a}s^{-m})
   (|c|^2+|d|^2+ d\overline{c}s^m+c\overline{d}s^{-m})-(\dots)
\end{equation}
\begin{equation}\label{phiphi4}
   o(s)= a\overline{b}\overline{c}d s^{2m}+(|a|^2\overline{c}d+|b|^2\overline{c}d+ 
   a\overline{b}|c|^2+a\overline{b}|d|^2)s^m 
\end{equation}
\begin{equation*}
  \ + (|a|^2c \overline{d} +|b|^2c \overline{d} + 
   \overline{a}b|c|^2+\overline{a}b|d|^2)s^{-m} + \overline{a}bc\overline{d}s^{-2m}-(\dots)+k\, .
\end{equation*}
 Note that \eqref{phiphi4} expands into 21 summands. We expand the square modulus $o(s)\overline{o(s)}$ into the sum of 21  diagonal terms,
    where each such term is a summand of
    \eqref{phiphi4} multiplied by its own complex conjugate, and  420 off diagonal terms.
The diagonal terms add up to $2K^2+k^2$, where $K$ is defined by
\begin{equation}\label{phiphi1}
    K^2= |a|^2|b|^2 |c|^2|d|^2+(|a|^4+|b|^4) |c|^2|d|^2 +|a|^2|b|^2(|c|^4+|d|^4) +(\dots)\, .
\end{equation}
Estimating the 420 off diagonal terms from above,
using that each of the 20 coefficients in
\eqref{phiphi4}
other than $k$  can be estimated in absolute value by $K$, gives by the reverse triangle inequality
  
\begin{equation}\label{phiphi5}\int\limits_{J} |o|^2 \, d|s| \ge 2K^2|J| + k^2|J|
-380K^2\left(\sum_{-2\le h,j\le 2, h^2+j^2\neq 0} |\int\limits_J s^{hn+jm} \, d|s| |\right)
\end{equation}
\begin{equation*}
-20|k|K\left(\sum_{-2\le h\le 2, h\neq 0} |\int\limits_J s^{hn} \, d|s||\right)
-20|k|K\left(\sum_{-2\le j\le 2, j\neq 0} |\int\limits_J s^{jm} \, d |s||\right)
\end{equation*}
Note that the integrands in \eqref{phiphi5} are powers of $s$, which oscillate with frequency at least $|J| 10^{-7}$ on $\T$ by Assumption \eqref{phiphihyp}. For any integral within \eqref{phiphi5}, we can reduce the domain of integration $J$ by any disjoint union of arcs on which the integrand undergoes exactly one full period, all without changing the value of the integral. Using this to reduce the domains of integration to an arc of size smaller than one full period, which by Assumption \eqref{phiphihyp} has arclength less than 
$10^{-6}|J|$, we  obtain
with trivial estimates on these integrals
\begin{equation}\label{phiphi6}
  \int\limits_{J} |o|^2 \, d|s| \ge (2K^2 + k^2- 380 K^2 10^{-6} -  40 |k|K 10^{-6})|J|\ge K^2|J| \, .
\end{equation}
With the assumed upper bound on $a,b,c,d,z,y,x,v$ we obtain 
\begin{equation}\label{phiphi7}
    \int\limits_{J} |o|^2 \, d|s| \le \int\limits_J 
    ||a s^m+ b|| c s^{-m}+ d|- |z s^n+y||x s^{-n}+v|| \, d|s|
\end{equation}
\begin{equation*}
   \times \sup_{t\in J}(|a s^m+ b|| c s^{-m}+ d|+ |z s^n+y||x s^{-n}+v|)
\end{equation*}
\begin{equation*}\label{phiphi8}
  \le 32 \int\limits_J 
    ||a s^m+ b|| c s^{-m}+ d|- |z s^n+y||x s^{-n}+v|| \, d |s| \, .
\end{equation*}
Inequalities \eqref{phiphi6} and \eqref{phiphi7}, along with \eqref{phiphi5}, show \eqref{phiphicon}, which completes  the proof of the  lemma.
\end{proof}

For $\eta  \in (0, \frac{1}{2})$ and  $m \in \mathbb{N}_0$, define
\begin{equation}\label{subs200}
    \Sigma_{\eta} ^m := \{ s\in E(\mu) \,:\, \,\forall n>m \text{ we have } 
\eta < \min(|A_n|,|B_n|,|\ti A_n|,|\ti B_n|) \} \, .
\end{equation}
\begin{lemma}\label{recur}
Let $\mu \in \mathcal{T}_{-}$, let $\eta  \in (0, \frac{1}{2})$ and $m \in \mathbb{N}_0$. If \eqref{convabhyp} holds, 
then  $|\Sigma_{\eta}^m| = 0$.
\end{lemma}

\begin{proof}
We proceed by contrapositive. Let $\mu, \eta,m $ be given and assume that  $|\Sigma_\eta^m|>0$.  Using $\Sigma_\eta^m\subset \Sigma_\eta^n$ for all $m<n$ and Lebesgue differentiation, 
there is $m_0>10^9$ and a subset $\Sigma\subset \Sigma_\eta^{m_0}\cap E(\mu)$ with
positive measure  such that
for all $s\in \Sigma$ and all $n>m_0$ we have
\begin{equation}\label{recur1}
    10^{10^{10^{10}}}L(\mu,s,n)\le \eta^8\, .
\end{equation}
Define $m_k:=4^km_0$ for $k\ge 1$.
For each $k\ge 0$, cover the set $\Sigma$ by a collection $\mathcal{J}_k$ of arcs $J$ with
\begin{equation}\label{recur3}
     |J|\le 10^{8}m_k^{-1}
 \end{equation}
such that every point of $\T$ is contained in at most two arcs of $\mathcal{J}_k$.
For $J\in \mathcal{J}_k$, we estimate with the local approximation Lemma \ref{approximation with A B C D},
applied at some point $s\in \Sigma\cap J$ for $C=10^{9}$ and for $n$ equals both $m_k$ and $m_{k+1}\le 10m_k$
and with the pointwise upper bound by $2$ on $\phi_n$, $\ti \phi_n$ and $A_n,B_n,\ti A_n,\ti B_n$
 \begin{equation}\label{recur4}
     \int\limits_{J} \left| |\phi_{m_k} \ti \phi_{m_k}| - 
     |\phi_{m_{k+1}} \ti \phi_{m_{k+1}}| \right| \, d |s| 
\end{equation}
\begin{equation*}
     \ge \int\limits_{J} \left| |(A_{m_k}s^{m_k}+B_{m_k})(\ti A_{m_k}s^{m_k}+\ti B_{m_k}) | -\right. 
\end{equation*}
\begin{equation*}  \left.|(A_{m_{k+1}}s^{m_{k+1}}+B_{m_{k+1}})(\ti A_{m_{k+1}}s^{m_{k+1}}+\ti B_{m_{k+1}}) | \right|\, d |s| -10^{-10}\eta^8|J|  \, .
\end{equation*}
With Lemma \ref{phiphi} for the obvious choice of parameters, using that
\eqref{phiphicon} is satisfied thanks to $m_{k+1}=4m_k$, we therefore obtain 
with \[
 |A_{m_k} B_{m_k}|, |\ti A_{m_k} \ti B_{m_k}|, |A_{m_{k+1}} B_{m_{k+1}}|, |\ti A_{m_{k+1}} \ti B_{m_{k+1}}| >\eta^2 
\]
that
 \begin{equation}\label{recur5}
     \int\limits_{J} \left| |\phi_{m_k} \ti \phi_{m_k}| - 
     |\phi_{m_{k+1}} \ti \phi_{m_{k+1}}| \right| \, d|s| \ge 10^{-3} \eta^8 |J|\, .
\end{equation}
Summing over $J\in \mathcal{J}_k$ and using that each point of $T$ is covered by at most two arcs of $\mathcal{J}_k$ while the union of arcs covers $\Sigma$, we obtain
 \begin{equation}\label{recur6}
     2\int\limits_{\T} \left| |\phi_{m_k} \ti \phi_{m_k}| - 
     |\phi_{m_{k+1}} \ti \phi_{m_{k+1}}| \right| \, d|s| \ge 10^{-3} \eta^8 |\Sigma|\, .
\end{equation}
It follows that the sequence $\phi_{m_k} ^* \ti \phi_{m_k}$ does not converge in
$L^1(\T)$, i.e., \eqref{convabhyp} fails. This proves the lemma.

\end{proof}

\begin{lemma}\label{lem:cty_phi_A_B}
For every $s \in \T$ and  $n\in \mathbb{N}_0$, 
\begin{equation}\label{phibyf}
    \left | \phi_{n+1} (s) - s\phi_{n} (s) \right| \leq 4 |F_{n+1}| 
\end{equation}
and, if $n\ge 1 $,
\begin{equation}\label{abbyf}
  |A_{n+1} - A_n|, |B_{n+1} - B_n| \leq 8|F_{n+1}| + \frac{1000}{n}  \, .  
\end{equation}
\end{lemma}
\begin{proof}
By the recursion equation \eqref{recursion phi} and then the bounds in Lemma \ref{lem:poly_bd_gen}, we have
\begin{equation}
\label{eq:diff_phi_s}
  \left | \phi_{n+1} (s) - s\phi_{n} (s) \right|
=  \left| s\phi_n(s) \left ( \frac{1}{\sqrt{1+|F_{n+1}|^2}}-1 \right ) + \frac{s^n \overline{F_{n+1}} \ti\phi_n^*(s) }{\sqrt{1+|F_{n+1}|^2}}\right|\le 4|F_{n+1}|\, ,
\end{equation}
where we used
\begin{equation} 
1- \frac{1}{\sqrt{1+|F_{n+1}|^2}}
\leq 
{\sqrt{1+|F_{n+1}|^2}} -1\le |F_{n+1}|\, .
 \end{equation}
   This proves \eqref{phibyf}.
 We have by Definition
 \eqref{def:A_B_n} and then the triangle inequality
   \[
   2 \left | |A_{n+1} | - |A_{n} | \right | = \left | |\phi_{n+1} (s) - \phi_{n+1} (s \gamma_{n+1})| - |s (\phi_{n} (s) - \phi_{n} (s \gamma_{n}))| \right | 
   \]
   \[\le 
  \left | \left ( \phi_{n+1} (s) - s\phi_n (s) \right ) - \left (  \phi_{n+1} (\gamma_{n+1} s) - s\phi_n (\gamma_{n} s) \right ) \right | \, . 
  \]
    \[\le 
  \left|  \phi_{n+1} (s) - s\phi_n (s)\right|+
  \left|  \phi_{n+1} (\gamma_{n+1} s) - s\phi_n (\gamma_{n+1} s)\right|+
  \left|  \phi_{n} (\gamma_{n+1} s) - \phi_n (\gamma_{n} s)\right|  \, . 
  \]
 By \eqref{eq:diff_phi_s} and the mean value theorem along the arc connecting $\gamma_{n+1}s$ with $\gamma_n s$, we obtain the bound
\begin{equation}\label{eq:bds_polys_cts_1}
2 \left | |A_{n+1} | - |A_{n} | \right | \le 8 |F_{n+1}| + |\gamma_{n+1} -\gamma_n| | \sup\limits_{z \in \T} |  \phi_n ' (z)| \, .  \end{equation}
 We have by Definition \eqref{gamman} of $\gamma$ that
\begin{equation}\label{smallf1}
    |\gamma_{n+1} -\gamma_n|\le \frac \pi n-\frac \pi{n+1}=\frac \pi {n(n+1)} \, .
\end{equation}
As for bounding $|\phi' (z)|$, by the Cauchy integral formula for a circle of radius $2/n$ centered at $z$ and then by the bounds of Lemma \ref{lem:poly_bd_gen},
\begin{equation}\label{smallf2}
     |\phi_n ' (z)|\le 
n\sup_{|y|\le 1+1/n} |\phi_n(y)|\le 100n\, .
\end{equation}
Combining \eqref{eq:bds_polys_cts_1}
with \eqref{smallf1} and \eqref{smallf2} proves to bound for $A_{n+1}-A_n$ in \eqref{abbyf}, and the 
  bound for $B_{n+1}-B_n$ follows in analogous fashion.
\end{proof}


We are ready to 
prove the second part of Theorem \ref{thm:conv_to_AB}.
Assume \eqref{convabhyp} and $\lim\limits_{n\to \infty}|F_n|=0$. Let us first define the set $E_0 (\mu)$.
By \eqref{convabhyp}, there is a monotone increasing
sequence $(m_k)$ such that $\phi_{m_k}^*\phi_{m_k}$
converges to $\overline{w(s)^{-1}}$
on a set $E_1(\mu)$ of full measure. Pick such a subsequence and set. Using the sets $\Sigma_\eta^m$ defined in \eqref{subs200}, define
\begin{equation}
    {E}_0(\mu):=
{E}_1(\mu) \setminus \left ( \bigcup\limits_{\eta \in 2^{-\N-2}} \bigcup\limits_{m\ge 1}  \Sigma_{\eta}^m \right ) \, .
\end{equation}
The set $E_0(\mu)$ has full measure in $\T$ by Lemma \ref{recur}.
Let $s\in E_0(\mu)$, and now assume \eqref{subshyp2} holds at this $s$. 

If $|w(s)|>1$, let
\begin{equation}\label{subs110}
\eta^{\frac 1 {100}}<10^{-12}|w(s)|^{-2}\min \{|w(s)|-1, 1 \}
\end{equation}
and if $|w(s)|=1$, let
\begin{equation}\label{subs190}
\eta^ { \frac{1}{100}}<10^{-12}|w(s)|^{-2}\, .
\end{equation}
Pick $n_0\geq 1000 \eta^{-2}$ large enough so that for all $n>n_0$ we have,  using first the assumption $F_n \to 0$, and  then Lebesgue differentiation, then Assumption \eqref{subshyp2},
\begin{equation}\label{subs111}
    |F_n|<10^{-10}\eta\, .
\end{equation}
\begin{equation}\label{subs112}
   e^{1000}L(\mu,s,n)< \eta^{10} \, ,
\end{equation}
\begin{equation}\label{subs113}
   |( \phi_n^*(s)\ti \phi_n(s))^2- \overline{w(s)^{-2}}|< \eta^{10} \, .
\end{equation}
As $s\not \in \Sigma_{10 \eta} ^{n_0}$, there exists $n_1>n_0$ such that
\begin{equation}\label{subs114}
\min(|A_{n_1}|,|B_{n_1}|,|\ti A_{n_1}|, |\ti B_{n_1}|)\le 10\eta\, .
\end{equation}
Assume first $|B_{n_1}|\le 10\eta$.
We then claim that for all $n\ge n_1$
\begin{equation}\label{subs115}
  |B_{n}|\le 10\eta  \, .
\end{equation}
We prove \eqref{subs115} by induction, the case $n=n_1$ being given. 
Assume \eqref{subs115} is proven for some $n\ge n_1$. We claim that
$|B_{n}|\le \eta$. Indeed, if not, then by Lemma \ref{abclem}, we must have 
\[
\eta^{5} \leq |\phi_n ^* (s) \ti \phi_n (s) - \overline{w(s)^{-1}}| \, ,
\]
which combined with \eqref{subs113}, implies
\begin{equation}\label{subs116}
|\phi_n ^* (s) \ti \phi_n (s) + \overline{w(s)^{-1}}| < \eta^{5} \, .
\end{equation}
On the other hand, by Lemma \ref{abtocdlem}, we have
\begin{equation}\label{subs117}
|A_n|, |\ti A_n| \geq (4|w(s)|)^{-1} \, .
\end{equation}
Furthermore by \eqref{abtocdcon3}, we also have
\begin{equation}\label{subs118.2}
|\ti B_n| \leq 100|w(s)|\eta \, .
\end{equation}
Combining estimates \eqref{subs117}, \eqref{subs116}, the bounds \eqref{subs115} and \eqref{subs118.2},
and approximation Lemma \ref{approximation with A B C D} with $C=10$, we obtain
\begin{equation}\label{subs119}
\left | \overline{A_n} \ti A_n + \overline{w(s)^{-1}} \right | < 100\eta ^{\frac 1 2} \, .
\end{equation}
On the other hand, using again the bound \eqref{subs115} on $|B_n|$ and \eqref{difference ab square} yields
\begin{equation} \label{subs120}
\left |-\overline{A_n} \ti A_n  + \overline{w(s) ^{-1}} \right | < 100 \eta ^{\frac 1 2} \, . 
\end{equation}
But \eqref{subs119} and \eqref{subs120} both contradict each other. Thus $|B_n| \leq \eta$. By Lemma \ref{lem:cty_phi_A_B},
we conclude $|B_{n+1}|\le 10\eta$.
This completes the proof of \eqref{subs115} by induction.
As $\eta$ was arbitrarily small, we obtain  
the conclusion \eqref{subscon3}.
Moreover, with Lemma \ref{smallablem} we obtain conclusion \eqref{subscon2}.
If $|\ti B_{n_1}|\le 10\eta$, then an analogous arguments gives
\eqref{subscon2}, and that $\ti A_n \ti B_n \to 0$, which then implies \eqref{subscon3} by \eqref{product ab}.

Now assume $|A_{n_1}|\le 10\eta$. Analogous arguments
give $|A_{n}|\le \eta$
for all $n\ge n_1$ and thus, by Lemma \ref{smallablem}
\begin{equation}\label{subs203}
\left| \phi_{n} ^* (s) \ti \phi_{n} (s)  +
\overline{w(s)^{-1}}\right|\le \eta ^{\frac 1 2} \, .
\end{equation}
However, as $s\in E_1(\mu)$, there exists $n_2>n_1$ with
\begin{equation}\label{subs202}
\left| \phi_{n_2} ^* (s) \ti \phi_{n_2} (s)  -
\overline{w(s)^{-1}}\right|\le \eta ^{\frac 1 2} \, .
\end{equation}
By the triangle inequality,
\eqref{subs203} and \eqref{subs202} contradict each other for $n=n_2$. Thus the case 
$|A_{n_1}|<10\eta$ is not possible.
By the analogous arguments, the case $|\ti A_{n_1}|<10\eta$ is not possible.
This proves Theorem \ref{thm:conv_to_AB}.

\section{Proof of Theorem \ref{zeros theorem}}

Let $\mu \in \mathcal{T}_{-}$ and $s\in E(\mu)$. 
Let $0< \epsilon\le  10^{-5}$.
With the Lebesgue differentiation theorem and elementary calculus, we find $n_0$ large enough so that for all $n>n_0$
\begin{equation}\label{zero1}
e^{10^{2}\epsilon^{-1}} L(\mu, s, n)\le |w(s)|^{-2}\, ,
\end{equation}
\begin{equation}\label{zero1.5}
   (10^{2}\epsilon^{-1})^{\frac 1 n}-1
   \le 10\epsilon^{-1}{n}^{-1}\, .
\end{equation}
\begin{equation}\label{zero1.7}
   10^4\epsilon^{-1}\le n\, ,
\end{equation}
Let $n> n_0$.
We first prove that \eqref{convergence of AB1}  
 implies         
 \eqref{zeros going to infinity1}.
Assume \eqref{convergence of AB1}. Using the upper bound \eqref{eq:bd_A_B} of $A_n,B_n$ by $2$, we obtain
\begin{equation}\label{zero3}
 \frac \epsilon {2}\le  | A_{n }|,|B_{n}| \le 2\, .
\end{equation}
The zeros of $A_n z^n+B_n$ are $n$ equidistant points on the circle of radius $r$ about the origin, where $|A_n|r^n=|B_n|$. 
Using \eqref{zero1.5} and \eqref{zero3}, we derive
\begin{equation}\label{zero4}
    |r-1|\le 10 \epsilon^{-1}n^{-1}\, .
\end{equation}
Hence there is  a zero $re^{i\theta}$
of $A_n z^n+B_n$
with distance at most 
\begin{equation}\label{zero5}
    |r-1|+10n^{-1}\le  20 \epsilon^{-1}n^{-1}
\end{equation}
from the point  $s\in \T$.
Consider the annular sector
$$K := \{q e^{i\zeta} \,: \, e^{-100 \epsilon^{-1}}  < q^n < e^{100 \epsilon^{-1}}  \text{ and } |\zeta-\theta| < \pi/n \}  \, .$$
The sector $K$  contains the zero $re^{i\theta}$.
Note that for $qe^{i\zeta}\in K$ we have 
\begin{equation}\label{zero5.5}
    |qe^{i \zeta} - s|   \leq  |q -1 | + | e^{i \zeta} - s |    \leq (e^{100\epsilon^{-1} n^{-1}} -1)  + 10 n^{-1}  \leq  10^3\epsilon^{-1} n^{-1} \, .
\end{equation}
In particular, the approximation Lemma \ref{approximation with A B C D} holds on $K$ with $C=10^4\epsilon^{-1}$.
For $qe^{i\zeta}$ on the pieces of the boundary of $K$ where $|\zeta-\theta|=\pi/n$, we have
\[|A_{n} q^n e^{i\zeta n}+B_n|\ge |B_n| \,  \]
because $re^{i\theta}$ is a zero of $A_nz^n+B_n$ and thus $A_n q^{n} e^{i \zeta n}$ and $B_n$ have the same argument.

For $qe^{i\zeta}$ on the piece of the boundary of $K$ where $q^n=e^{100\epsilon^{-1}}$, we have 
\[|A_{n} q^n e^{i\zeta n}+B_n|\ge e^{100\epsilon^{-1}}|A_n| -|B_n|\ge |B_n| \]
by \eqref{zero3}.
On the piece of the boundary where $q^n=e^{-100\epsilon^{-1}}$, we have 
\[|A_{n} q^n e^{i\zeta n}+B_n|\ge |B_n|-e^{-100\epsilon^{-1}}|A_n| \ge |B_n| \]
again by  \eqref{zero3}.
Hence, for $z$ on the boundary of $K$,
\begin{equation}\label{zero10}
    |A_n z^n + B_n|\ge 2^{-1} \epsilon\, .
\end{equation}
This together with local approximation Lemma \ref{approximation with A B C D} yields
\[
|An z^n + B_n| > |\phi_n(z) - (A_n z^n + B_n)|
\]
for $z \in \partial K$. By Rouch\'e's theorem, $\phi_n$
also has a zero in $K$.
With \eqref{zero5.5}, this completes the proof of 
 \eqref{zeros going to infinity1}.

We turn to the proof that \eqref{zeros going to infinity2} implies \eqref{convergence of AB2} and thus assume in addition $e^{\epsilon^{-1}} \geq |w(s)|$. 
Assume \eqref{zeros going to infinity2} and pick a zero $z$ of $\phi_n$ with 
\begin{equation}\label{zero20}
    |z-s|\le \epsilon^{-1}n^{-1}\, .
\end{equation}
With $\phi_n(z)=0$, the approximation
Lemma \ref{approximation with A B C D},  and the choice of $n$ that yields \eqref{zero1}, we conclude
\begin{equation}\label{zero300}
    |A_nz^n-B_n|\le e^{-10\epsilon^{-1}}\, .
\end{equation}
Squaring and using the triangle inequality gives
\begin{equation}\label{zero21}
    |A_nB_nz^n|\ge |A_n z^n|^2+|B_n|^2 - e^{-20\epsilon^{-1}}\, .
\end{equation}
Dividing by $z^n$ and using that $|z^n|, |z|^{-n}<e^{2\epsilon^{-1}}$, 
\begin{equation}\label{zero22}
    |A_nB_n|\ge (|A_n|^2+|B_n|^2)e^{-4\epsilon^{-1}} - e^{-18\epsilon^{-1}} \, .
\end{equation}
Inequality \eqref{difference ab square} together with the upper bound $|\ti A_n|, |\ti B_n|\le 2$ gives the lower bound
\begin{equation}\label{zero200}\max(|A_n|,|B_n|)\ge 10^{-1}|w(s)|^{-1}\, ,
\end{equation}
which can be applied to 
\eqref{zero22} to obtain with the upper bound  $e^{\epsilon^{-1}}$ by
 $|w(s)|$,
\begin{equation}\label{zero23}
    |A_nB_n| \geq 10^{-1} |w^{-1}| e^{-4 \epsilon^{-1}} - e^{-18 \epsilon^{-1}} \geq 10^{-1} e^{-5 \epsilon^{-1}} - e^{-18\epsilon^{-1}} \ge e^{-10\epsilon^{-1}} \, .
\end{equation}
This gives \eqref{convergence of AB2} and completes the proof of Theorem \ref{zeros theorem}.

\section{Proof of Theorem \ref{thm:lacunary}}

We begin with three preliminary lemmas to further compare the sizes of two waves with lacunary frequencies. 

\begin{lemma}\label{fourlem}
    Let $a,b,c,d$ be complex numbers bounded in absolute value by $2$. Let $J$ be an arc on $\T$.
    Let $m,n$ be  positive integers with  
\begin{equation}\label{fourhyp}
        10^7|J|^{-1}\leq m,n,|m-n|,|2m-n|, |m-2n|\, .
    \end{equation}
   Then
   \begin{equation}
       \label{fourcon}
    \int\limits_J (|a s^m+ b|- |c s^n+d|)^2\, d|s|\ge 
10^{-6}\min(|ab|,|cd|)^8|J|\, .
   \end{equation}
\end{lemma}
\begin{proof}
Define
\begin{equation}\label{four1}
   o(s):= |a s^m+ b|^2- |c s^n+d|^2 \, .
\end{equation}
Applying Lemma \ref{phiphi} with the obvious parameters
gives
\begin{equation}\label{four100}
    \| o\|_{L^1(J, \, d|s|)}\ge 10^{-2}\min(|ab|,|cd|)^4|J|\, .
\end{equation}
Applying Cauchy Schwarz and the assumed upper bound on $a$, $b$, $c$, and $d$ gives
\begin{equation}\label{four6}
    \|o\|_{L^1(J, \, d|s|)}\le 
     \||a s^m+ b|+ |c s^n+d|\|_{L^2(J, \, d|s|)}
    \||a s^m+ b|- |c s^n+d|\|_{L^2(J, \, d|s|)}
\end{equation}
\begin{equation*}
    \le 8|J|^{\frac 12}
    \||a s^m+ b|- |c s^n+d|\|_{L^2(J, \, d|s|)}\, .
\end{equation*}
Inequalities \eqref{four100} and \eqref{four6} imply\eqref{fourcon} and complete the proof of the lemma.
\end{proof}

\begin{lemma}\label{sin}
Let $a$ and $m$ be positive, let $b,c$ be nonnegative with $b\le a$, and let $\zeta$ be real.
Let $J$ be an
interval on $\R$ with  $|J|>10 m^{-1}$.  Then
\begin{equation}\label{sincon}
    \int\limits_{J} (\sqrt{a+b\sin(m(x-\zeta))}-c)^2\, dx \ge 10^{-2} a^{-1} b^2 |J|\, .
\end{equation}
\end{lemma}
\begin{proof}
   After translation and dilation, we assume $m=1$ and $\zeta=0$. Dividing by $a$ if necessary and replacing $b/a$ by $b$ and $c/\sqrt{a}$ by $c$, we may assume $a=1$.
The left side of 
\eqref{sincon} is then larger than 
the $L^2$ norm squared of 
the projection of 
$\sqrt{1+b\sin x}$ onto the 
orthogonal complement of $\mathbf{1}_J$. To obtain a lower bound on this projection, pick  a disjoint union $I\subset J$ of
full periods of $\sin(x)$ with $2|I|>|J|$, which exists by assumption on  $J$, and pair
with the function function $\mathbf{1}_{I}\sin$, which has mean zero is hence in the orthogonal complement of constant functions. Because $\mathbf{1}_{I}$ has $L^2$ norm $2^{-\frac 12}|I|^{\frac 12}$, we thus estimate the left side of 
\eqref{sincon} from below by
\begin{equation}\label{sin4}
2|I|^{-1}\left(\int\limits_{ I}\sin(s)\sqrt{1+b\sin(s)}\, ds\right)^2 \, .
\end{equation}
We divide the domain of integration in \eqref{sin4} into the set $I_+\subset I$  where $\sin$
is positive and $I_-=I\setminus I_+$.
Using for $0\le x\le 1$ the elementary inequality
\begin{equation}\label{sin5}
    1\le 1+x\le \left(1+\frac x2\right)^2\, ,
\end{equation}
we estimate with the assumption $b\le 1$
\begin{equation}\label{sin6}
\int\limits_{ I_+}\sin(s)\sqrt{1+b\sin(s)}\, ds
\ge \int\limits_{ I_+}\sin(s)\, ds\, ,
\end{equation}
\begin{equation}\label{sin8}
\int\limits_{ I_-}\sin(s)\sqrt{1+b\sin(s)}\, ds
\ge \int\limits_{ I_-}\sin(s)\left(1+ 2^{-1}b \sin(s)\right)\, ds\, .
\end{equation}
Adding \eqref{sin6} and \eqref{sin8} and using that $\sin(s)$ has integral zero on $I$  gives
\begin{equation}\label{sin9}
\int\limits_{ I}\sin(s)\sqrt{1+b\sin(s)}\, ds
\ge  
2^{-1}b\int\limits_{ I_-}\sin(s)^2\, ds = 8^{-1}b |I|\, ,
\end{equation}
where we used that $\mathbf{1}_{I_-}\sin^2$ has mean $1/4$ over each period.
Inserting \eqref{sin9} into \eqref{sin4} and using $2|I|\ge |J|$
proves \eqref{sincon}
and completes the proof of the lemma.
\end{proof}

\begin{lemma}\label{unbalanced}
    Let $a,b,c,d$ be complex numbers bounded in absolute value by $2$ and assume
    \begin{equation}\label{unbalancedhyp}\min(|a|,|b|)|\ge 10^2\min(|c|,|d|)\, .
    \end{equation}
     Let $J$ be an arc on $\T$.
    Let $m,n$ be  positive integers with  
\begin{equation}\label{unbalancedhyp1}
        10^3|J|^{-1}\le m,n\, .
    \end{equation}
   Then 
\begin{equation}
       \label{unbalancedcon}
    \int\limits_J (|a s^m+ b|- |c s^n+d|)^2\, d|s|\ge 
10^{-4} \min(|a|,|b|)^2|J|\, .
   \end{equation}
\end{lemma}
\begin{proof}
Define
\begin{equation}\label{unbalanced1.1}
    u(s):=|a s^m+ b|- \max(|c|,|d|)\, ,\ 
 v(s):=\max (|c|,|d|)
    - |c s^n+d|
\end{equation}
Then we estimate the left side of \eqref{unbalancedcon}
 by
\begin{equation}
\label{unbalanced2}
    \|u-v\|^2\ge (\|u\|-\|v\|)^2
    \ge \|u\|(\|u\|-2\|v\|)\ge
    \|u\|(\|u\|-2\min(|c|,|d|)|J|^{\frac 12})\, ,
\end{equation}
where the norms are in $L^2(J, \, d|s|)$
and in the last step we have used that for all $s\in \T$
\[
    -\min(|c|,|d|)=\max (|c|,|d|)
    - |c|-|d| \le v(s)\le \max (|c|,|d|)
    - ||c|-|d|| =\min(|c|,|d|)\, .
\]
Noting that, writing $s = e^{i x}$ for $x \in [0, 2 \pi)$, for some real $\zeta$, we have
\begin{equation}\label{unbalanced4}
 |a s^m+ b|=\sqrt{a^2+b^2+2|a b|\sin(m(x-\zeta))}\,  . \end{equation}
By Lemma \ref{sin}, which takes Assumption \eqref{unbalancedhyp1}, to estimate
\begin{equation}\label{unbalanced5}
\|u\| \ge 5^{-1}|ab|(|a|^2+|b|^2)^{-\frac 12}|J|^{\frac 12}
\ge 10^{-1}\min(|a|,|b|)|J|^{\frac 12}
\, .
\end{equation}
Inserting \eqref{unbalanced5} into \eqref{unbalanced2}
and using the assumption \eqref{unbalancedhyp}
proves \eqref{unbalancedcon}
and thus completes the proof of the lemma.
\end{proof}

We turn to the proof of Theorem \ref{thm:lacunary}. 
Let $\mu\in \mathcal{T}_-$ be given and
assume \eqref{convabhyp} and that $\|F\|_{\ell^2(\mathbb{N}_0)}$ is finite.
Let  $(n_k)$ be a lacunary sequence.
Define the set 
\begin{equation}
S:=\bigcup_{\eta\in 2^{-\mathbb{N}_0-2}} S_\eta\, ,
\end{equation}
where
\begin{equation}
    S_\eta:=\{s\in \T: \forall m\in \mathbb{N}_0: \exists k>m:\eta\le \min(|A_{n_k,s}|,|B_{n_k,s}|,|\ti A_{n_k,s}|, |\ti B_{n_k,s}|)\}\, .
\end{equation}
We claim 
\begin{equation}\label{sclaim}
    |S|=0\, .
\end{equation}
Assume for now that \eqref{sclaim} is true.
Define 
\begin{equation}
E_2(\mu)=E(\mu)\setminus S
\end{equation}
and note that $E_2(\mu)$ has full measure by Lebesgue differentiation and \eqref{sclaim}. Let $s\in E_2(\mu)$. Because $s\notin S$, we have
\begin{equation}\label{tfive200}
    \lim\limits_{k\to \infty}\min(|A_{n_k,s}|,|B_{n_k,s}|,|\ti A_{n_k,s}|, |\ti B_{n_k,s}|)=0\, .
\end{equation}
By Lemma \ref{abtocdlem}, we obtain \eqref{subshyp1}.
Applying Theorem \ref{thm:conv_to_AB}
gives \eqref{subscon1}, which in turn gives  the desired conclusion of Theorem \ref{thm:lacunary}.

It remains to prove the claim \eqref{sclaim}. Assume to get a contradiction that $|S|>0$. Pick an $\eta\in 2^{-\mathbb{N}_0-2}$ such that $|S_\eta|>0$. 
Let $L$ be a parameter such that $(n_k)$ is $L$-lacunary and  assume $L=2^{2/\lambda}$ with an even integer $\lambda >2$ by 
lowering $L$ if necessary.
Decompose
\begin{equation}\label{tfive50}
\mathbb{N}_0=\bigcup_{\mu \in \mathbb{Z}: 0 \le \mu < 8\lambda}N_\mu\, ,
\end{equation}
where $N_\mu$ is the set of all $l\in \mathbb{N}_0$ such that there exists an integer $m$ with
\begin{equation}
    2^{8m+ \mu/\lambda }\le l< 2^{8m+ (\mu+1)/\lambda }\, .
\end{equation}
For each integer $\mu$ with $0\le \mu<8\lambda$, let 
$(n_k^{\mu})$ be the monotone subsequence 
of $(n_k)$ whose image is the intersection of the image of $(n_k)$ with  $\N_\mu$. As $(n_k)$ is $L$-lacunary, $(n_k^{(\mu)})$ is  $2^7$-lacunary. There is an integer $\mu$ with  $0\le \mu< 8\lambda$ such that the set
\[
     S^1:=\{s\in \T: \forall m\in \mathbb{N}_0: \exists k>m:\eta\le \min(|A_{n_k^{(\mu)},s}|,|B_{n_k^{(\mu)},s}|,|\ti A_{n_k^{(\mu)},s}|, |\ti B_{n_k^{(\mu)},s}|)\}
\]
has positive measure. Pick such a $\mu$ and let $(m_k)$ denote the sequence $n_k^{(\mu)}$.

The set $S^2:= S^1 \cap E(\mu)$ has positive measure. 
Pick an $n_0>10^{10}$ and a
further subset $S^3$ of $S^2$ of positive measure such that
for all $n>n_0$ and all $s\in S^3$ we have
\begin{equation}\label{tfive100}
    10^{10^{10}} L(\mu,s,n)\le \eta^{16}.
\end{equation}
Using inner regularity,
pick a compact subset
$S^4$ of $S^3$ of positive measure.

We define recursively 
for $\ell\in \mathbb{N}_0$  a number $k_{\ell-1}$  
and a  finite collection $\mathcal{I}_\ell$ of arcs of length less than $10^7m_{k_{\ell-1}}^{-1}$ with
\begin{equation} \label{tfive104}
     S^4\subset \bigcup\limits_{I \in \mathcal{I}_\ell} I 
\, ,
\end{equation}
\begin{equation}\label{tfive106}
\|\sum_{I \in \mathcal{I}_\ell}\mathbf{1}_I\|_{L^\infty(\T)}\le 2\, .
\end{equation}
Pick $k_{-1}$ large enough so that $m_{k_{-1}}>n_0$.
Let $\ell\in \mathbb{N}_0$
and assume we have already defined $k_{\ell-1}$. 
For every point $s\in S^4$,
using the definition of $S_\eta$, find a $k(s)>{k_{\ell-1}}$ such that
\begin{equation}\label{tfive108}
\eta\le \min(A_{m_{k(s)},s},B_{m_{k(s)},s},\ti A_{m_{k(s)},s}, \ti B_{m_{k(s)},s})\} \, ,  \end{equation}
and let $J_s$ be the 
open
arc with length
$10^{7}m_{k(s)}^{-1}$ centered at $s$. By compactness, there is a finite cover of $S^4$ by sets $J_s$ with $s\in S^4$. Throwing away succesively superfluous covering arcs, that is arcs which themselves are covered by two other arcs of the collection,  we find a finite set $T_\ell\subset S^4$ of centers such that, defining
\begin{equation}
  \mathcal{I}_\ell:= \{J_s: s\in T_\ell\}  \, ,
\end{equation}
we have \eqref{tfive104} and \eqref{tfive106}.
Define $k_\ell:=\max_{s\in T_\ell}k(s)+1$. 

Thanks to the separation provided by the sequence $(k_\ell)$,
the sets $\mathcal{I}_\ell$ are pairwise
disjoint. Define $\mathcal{I}=\bigcup_{\ell \in \mathbb{N}_0} \mathcal{I}_\ell$.
Define $\mathcal{I}^{(j)}$ to be the set of all 
$J\in \mathcal{I}$
of length $10^7 m_j^{-1}$ and note that for each $j$, if $\mathcal{I}^{(j)}$ is nonempty, there is a unique $\ell$ such that 
$\mathcal{I}^{(j)}
\subset \mathcal{I}_\ell$.

Let $j>2$ and let $J\in \mathcal{I}^{(j)}$ 
and write  $l=m_j$, $h=4m_j$. We estimate
with the triangle inequality and the binomial formula and then \eqref{eq:ortho_polys_det2}:
\begin{equation}\label{tfive6}
        \frac 12 \int\limits_J \left| \phi_l^* \phi_h + \ti\phi_l^*\ti\phi_h \right|  \, d |s| 
\end{equation}
\begin{equation}\label{tfive7}
   \le \frac 14 \int\limits_J  \left[\phi_l|^2+|\phi_h|^2 - (|\phi_l|-|\phi_h|)^2\right]+ 
   \left[|\ti \phi_l|^2+|\ti \phi_h|^2 - (|\ti \phi_l|-|\ti \phi_h|)^2\right] \, d|s|
\end{equation}
\begin{equation}\label{tfive8}
  \le  |J|-  \frac 1 4 \int\limits_J (|\phi_l|-|\phi_h|)^2+ 
     (|\ti \phi_l|-|\ti \phi_h|)^2 \, d|s|
     \le  |J|- \frac 1 4 \int\limits_J | (|\phi_l|-|\phi_h|)^2 \, d|s|\, .
\end{equation}
We use the approximation Lemma \ref{approximation with A B C D} at the center of $J$ with $C=10^8$, the upper bound given by \eqref{eq:ortho_polys_det2} on the one-sided orthogonal polynomials, and \eqref{tfive100}, to estimate \eqref{tfive8} from above by
\begin{equation}\label{tfive9}
  \le   |J|- \frac 1 4 \int\limits_J (|A_ls^l+B_l|-|A_hs^h+B_h|)^2 \, d|s|+ 10^{-200}\eta^8|J|\, .
\end{equation}

If $\min(|A_h|,|B_h|)\ge 10^{-2}\eta$ then we 
use Lemma \ref{fourlem},
and if $\min(|A_h|,|B_h|)< 10^{-2}\eta$,
we use estimate \eqref{tfive108} with $m_k(s) = m_k$ and Lemma \ref{unbalanced} to estimate \eqref{tfive9} from above by
\begin{equation}\label{tfive10} 
(1-10^{-100}\eta^8)|J|  \, .
\end{equation}
Using  $\log x\le x-1$ for positive $x$,
 we conclude from the bound \eqref{tfive10} for \eqref{tfive6} that
\begin{equation}\label{tfive11}
         \int\limits_J \log\left( \frac 12| \phi_l^* \phi_h + \ti\phi_l^*\ti\phi_h |\right) \, d|s|  \le 
-10^{-100}\eta^8|J|  \, .
\end{equation}
Summing  \eqref{tfive11} over $\mathcal{I}^{(j)}$
and using \eqref{tfive106}, we obtain
\begin{equation}\label{tfive22}
     -2 \int\limits_{\bigcup\limits_{I \in  \mathcal{I}^{(j)}} I} \log \left( \frac 12\left| \phi_{l}^* \phi_{h} + \ti \phi_{l}
     \ti \phi_{h} \right|
     \right) \, d|s|\ge 
10^{-110}\eta^8
\sum_{J\in \mathcal{I}^{(j)}}
|J| \, .
\end{equation}
By Plancherel, Lemma \ref{lem:plancherel}, this implies
 \begin{equation}
\sum_{m_j<l\le 4m_j}\log (1+|F_l|^2)
   \ge 10^{-120}\eta^8
\sum_{J\in \mathcal{I}^{(j)}}
|J|  \, .
\end{equation}
Summing over all $\mathcal{I}^{(j)}\subset \mathcal{I}_\ell$, using $8$-lacunarity of $(m_k)$ to obtain
disjointness of the intervals $[m_j, m_{j+1}]$ and using that $\mathcal{I}_\ell$ covers $S^4$, 
gives
\begin{equation}
\sum_{m_{k_\ell}<l\le  m_{k_{\ell+1}}}\log (1+|F_l|^2)
   \ge 10^{-130}\eta^8
|S^4|  \, .
\end{equation}
This implies
\begin{equation}
\sum_{m_{k_\ell}<l\le  m_{k_{\ell+1}}} |F_l|^2
   \ge 10^{-140}\eta^8 |S^4|  \, .
\end{equation}
This holding for infinitely many $\ell$ contradicts square summability of 
the sequence $F$.
This completes the proof of Theorem \ref{thm:lacunary}.

\section{Proof of  Theorem \ref{thm:su2}}

We first discuss Part \eqref{thm:su2_part1}. 
Let $(a,b)$ as in Part \eqref{thm:su2_part1}.
By  \cite[Lemma 3.7]{tsai} and \cite[Theorem 11]{QSP_NLFA}, there exists a unique $(F_n)\in \ell^2(\mathbb{N}_0)$
such that $(a,b)$ is the nonlinear Fourier series of $(F_n)$. Define $\phi_n$, $\ti \phi_n$,
$a_n$,$b_n$, by 
\eqref{recursion phi}, \eqref{recursion ti phi}, and \eqref{eq:NLFT_defn_intro}. By \cite{tsai}, $(a_n ^*, b_n) \to (a^*, b)$ in $H^2 (\D) \times H^2 (\D)$.  

 Define the measurable function $w$ on $T$ by 
\begin{equation}\label{measure with NLFT}
        w := \frac{1}{(a^*-b)(a+b^*)} \, .
    \end{equation}
As $|a|^2+|b|^2=1$ on $\T$ and 
$\|b\|_\infty^2<\frac 12$, we see that the quotient is well defined and $|w|$ is bounded above
 and thus in $L^1(\T)$, meaning the measure
\begin{equation}
    \mu := w \frac{d |z|}{2 \pi} 
\end{equation}
is well-defined.

We first check that $\mu \in \mathcal{T}$, namely by showing that for each $n \geq 0$, the functions
\begin{equation}\label{eq:monic_polys_def}
\Phi_n := a_n ^* (0)^{-1} \phi_n  \, , \qquad \ti \Phi_n  := a_n ^* (0)^{-1} \ti \phi_n  
\end{equation}
are the unique monic left and right monic orthogonal polynomials of degree $n$ for $\mu$. To see the polynomials are monic of degree $n$, by \eqref{eq:NLFT_defn_intro}, we have
\begin{equation}\label{eq:phi_to_ab}
z^{n} \phi_n ^*(z) = a_n ^* (z) + b_n (z)  \, .
\end{equation}
Because $F$ is supported on $[1,n]$, then $b_n$ also has frequency support on $[1,n]$, and so vanishes at $0$. Evaluating both sides of \eqref{eq:phi_to_ab} at $0$ yields $\phi_n$ has leading coefficient $a_n ^* (0)$, which is positive. Similar reasoning yields $\ti \phi_n$ also has leading coefficient $a_n ^* (0)$, and so both $\Phi_n$ and $\ti \Phi_n$ are monic polynomials of degree exactly $n$.

We now check left and right orthogonality of $\Phi_n$ and $\ti \Phi_n$. It suffices to prove left and right orthogonality of $\phi_n$ and $\ti \phi_n$. As we have the  ordered product \eqref{gproduct1}, convergent in $L^2(\T)$, 
we may write
\begin{equation}\label{eq:matrix_nlft_jump}
\begin{pmatrix}
    a & b \\ - b^* & a^*
\end{pmatrix} = \begin{pmatrix}
    a_n & b_n \\ - b_n ^* & a_n ^*
\end{pmatrix}\begin{pmatrix}
    a_+ & b_+ \\ - b_+ ^* & a_+ ^*
\end{pmatrix}
\end{equation}
where $(a_n,b_n)$ is the NLFS of $F \mathbf{1}_{[0,n]}$ and coincides with previous definition of $(a_n,b_n)$ and $(a_+,b_+)$ is the NLFS of $F \mathbf{1}_{[n+1,\infty]}$. 
Then, inverting the right most matrix in \eqref{eq:matrix_nlft_jump} and using the fact that all the matrices in \eqref{eq:matrix_nlft_jump} are in $SU(2)$ and hence have determinant $1$, we have
\[
a_n = aa_+^*+bb_+^*\, , \qquad b_n=-ab_++ba_+ \, .
\]
Hence, with \eqref{eq:NLFT_defn_intro},
$$\phi_n(z) = z^n(a_n+b_n^*)=z^na_+^*(a+b^*)+z^nb_+^*(b-a^*) \, .$$
A similar formula holds for $\ti \phi_n$ if we replace $b$ and $b_+$ by $-b$ and $-b_+$, respectively.
We compute
\begin{equation}\label{eq:ortho_comp}
    \langle \phi_n , z^m \rangle_{\mu} = \int\limits_{\T} \frac{z^{n-m} a_+^*}{a^*-b} \frac{d z}{2 \pi i z} - \int\limits_{\T} \frac{z^{n-m} b_+ ^*}{a+b^*}\frac{d z}{2 \pi i z} \, .
    \end{equation}
We show this vanishes whenever $m \leq n-1$. Because $a^*$ is outer on $\D$ and is bounded below by $\sqrt{1-\|b^2\|_{L^{\infty}}^2} > \frac{1}{\sqrt{2}}$ on $\T$, then $(a^*)^{-1}$ is holomorphic on $\D$ with $L^{\infty}$ norm strictly less than $\sqrt{2}$. By the maximum principle, we have $|a^*| > \frac{1}{\sqrt{2}}$ all throughout $\D$. In particular, on $\D$ we have 
\[
|a^* -b| \geq |a^*| - |b| > \frac{1}{\sqrt{2}} - \frac{1}{\sqrt{2}} = 0 \, ,
\]
i.e., $a^*-b$ does on vanish on $\D$.
Then the first integral in \eqref{eq:ortho_comp} vanishes because 
$\frac{a_+^*}{a^*-b}$ is holomorphic in $\D$.
To see the second integral in \eqref{eq:ortho_comp} vanishes, we look at its conjugate
\begin{equation}\label{eq:second_term}
\int\limits_{\T} \frac{z^{m-n} b_+ }{a^* +b} \frac{dz}{2 \pi i z} \, ,
\end{equation}
 noting the measure $\frac{dz}{2 \pi i z}$ is invariant under conjugation. Note that $b_+$ has a zero of order at least $n+1$, which can be seen by noting the shift formula $(a_+, b_+) = (a', b' z^{n+1})$, where $(a',b')$ denotes the shift of the sequence $F\mathbf{1}_{[n+1, \infty)}$ by $n+1$ nodes to the left. Because $\frac{b_+}{a^*+b}$ is holomorphic in $\D$ and has a zero of order at least $n+1$ at $0$ and $n-m\leq n$, then $\frac{z^{m-n} b_+ }{a^* +b}$ has a zero of order at least $1+m$ when $m <n$. Thus $\langle \phi_n , z^m \rangle_{\mu} = 0$ for $n > m \geq 0$. Thus $\phi_n$ is a left orthogonal polynomial with respect to $\mu$. By \eqref{mumubar}, $\ti \phi_n$ is right orthogonal for $\mu$ if and only if $\ti \phi_n$ is left orthogonal for $\overline{\mu}$, which then follows by virtually the same argument as above, just replace $b$ and $b_n$ by $-b$ and $-b_n$, respectively. Thus $\Phi_n$ and $\ti \Phi_n$ are left and right monic orthogonal polynomials with respect to $\mu$. 

We turn to proving uniqueness of the left/right orthogonal polynomials for $\mu$. By Lemma \ref{unique}, uniqueness of the left orthogonal polynomials will follow from checking \eqref{eq:ortho_comp} is nonzero when $m=n$. The second term on the right side of \eqref{eq:ortho_comp} still vanishes for $m=n$ by the arguments surrounding \eqref{eq:second_term}. By the mean value theorem and then the fact that $b$ vanishes at $0$, the first term on the right of \eqref{eq:ortho_comp} equals
    \begin{equation}\label{eq:first_int_nonvanishing}
    \frac{a_+ ^* (0)}{a ^* (0) - b (0)} = \frac{a_+ ^* (0)}{a ^* (0) } \, , 
    \end{equation}
    which is positive by \eqref{eq:a_positive}.
Thus for each $n$, $\Phi_n$ is the unique left monic orthogonal polynomials for $\mu$. A similar argument yields that $\ti \Phi_n$ is the unique monic right orthogonal polynomial of $\mu$. Thus $\mu \in \mathcal{T}$.

Noting that 
\[
a_n ^* (0) = \prod\limits_{1 \leq j \leq n} (1 + |F_j|^2)^{- \frac 1 2} \, ,
\]
\cite[Lemma 2.1]{tsai}, we note that the definitions \eqref{recursion phi} \eqref{recursion ti phi} of $\phi_n$ and $\ti \phi_n$ imply the Szeg\H o recursion \eqref{eq:Szego_recur_intro} with $\ti F_n = - \ti F_n$. Thus $\mu \in \mathcal{T}_{-}$, and $\phi_n$ and $\ti \phi_n$ are the normalized left and right orthogonal polynomials associated to $\mu$.
   
   Next, we show that $\phi_n$, $\ti \phi_n$ and $w$ satisfy \eqref{convabhyp}. We use \eqref{eq:NLFT_defn_intro} to write for $z \in \T$,
    \[
    \phi_n ^* \ti \phi_n = \overline{ z^{-n} (a_n + b_n ^*) } z^{-n}(a_n - b_n ^*) = (a_n ^* + b_n)(a_n - b_n ^*) \, .
    \]
    Because $F \in \ell^2 ([1, \infty))$, then $(a_n, b_n) \to (a, b)$ in $H^2 (\D^*) \times H^2 (\D)$ \cite{tsai, QSP_NLFA}, and so $\phi_n ^* \ti \phi_n$ 
    converges to 
\begin{equation}\label{abab}
        (a^* + b)(a-b^*)
    \end{equation}
    in $L^1 (\T)$.
    The expression \eqref{abab} however is $\frac{1}{\overline{w}}$ by definition of $w$. This shows 
\eqref{convabhyp}.

Finally, let $(n_k)$ be a lacunary sequence.
By Theorem \ref{thm:lacunary}, $(\phi_{n_k} ^* \ti\phi_{n_k})^2$ converges almost everywhere to $\overline{w^{-1}}$. With \eqref{eq:NLFT_defn_intro} and \eqref{measure with NLFT}, we conclude \eqref{su2con}.
This completes the proof of Part \eqref{thm:su2_part1}
of Theorem \ref{thm:su2}.

We turn to the proof of Part \eqref{thm:su2_part2}.
Let $r> 0$ and define  $(F_{r,n})$ to be the sequence whose only nonzero entry is $F_{r,1}=r$.
The nonlinear Fourier series of
$(F_{r,n})$
is given by $a_{r,0}=1$, $b_{r,0}=0$,
and for $n\ge 1$
\begin{equation}\label{counter1}
    a_r(z)= a_{r,n}(z)=(1+r^2)^{-\frac 12}\, ,
\end{equation}
\begin{equation}\label{counter2}
     b_r(z)= b_{r,n}(z)=(1+r^2)^{-\frac 12}rz\, .
\end{equation}
We observe that $a_r^*$ is constant and thus outer 
and 
\begin{equation}\label{counter3}
    \|b_r\|_{L^\infty(\T)}^2= \frac {r^2}{1+r^2}= \frac 1{1+1/r^2}\, .
\end{equation}
Thus, for $r<1$, the pair $(a_r,b_r)$ satisfies the assumptions of Part \ref{thm:su2_part1} of Theorem \ref{thm:su2} and, following  
\eqref{measure with NLFT}, the density
\begin{equation}\label{counter30}
      w_r(s) :=
      \frac {1+r^2}{(1-rs)(1+rs^{-1})}=
      \frac {1+r^2}{(1-r^2 -2 ri\Im(s))}
\end{equation}
defines an absolutely continuous measure $\mu_r $ on $\T$ in the class $\mathcal{T}_-$ with normalized one-sided orthogonal polynomials given by 
$\phi_0^{(r)}=\ti \phi_0^{(r)}=1$ and, following \eqref{eq:NLFT_defn_intro},
\begin{equation}\label{counter4}
   \phi_{r,n}(z) =z^n (a_{r,n}(z)+b_{r,n}^*(z))=(1+r^2)^{-\frac 12}(z^n+rz^{n-1})\, ,
\end{equation}
\begin{equation}\label{counter5}
    \ti \phi_{r,n}(z)=z^n (a_{r,n}(z)-b_{r,n}^*(z))=(1+r^2)^{-\frac 12}(z^n-rz^{n-1})\, .
\end{equation}
As $r$ tends to $1$ from below, 
\eqref{counter4} and \eqref{counter5} have limits 
$\phi_{1,n}(z)$ and $ \ti \phi_{1,n}(z)$, which satisfy the relation
\begin{equation}\label{counter6}
\Lambda(\phi_{1,n}{\ti \phi_{1,m} ^*})=\delta_{n,m}
\end{equation}
with the principal value distribution
\begin{equation}\label{counter7}
\Lambda(\psi)=\lim_{r\nearrow 1}
\int\limits_{\T} \psi w_r \, \frac{d |z|}{2 \pi}\, .
\end{equation}
The orthogonality relations \eqref{counter6}
allows to compute  $\Lambda(s^m)$ by expressing $s^m$ as linear combinations of the one-sided orthogonal polynomials $\{\phi_{j}\}_{j=0}^m$ . Therefore,
 \eqref{counter6}
allows to determine the Fourier coefficients of $\Lambda$ and thus determines $\Lambda$ uniquely. In particular, as $\Lambda$ has outside its singularities at $\pm 1 $ non-integrable density, it is not given by integration against a measure. Hence there is no measure in 
$\mathcal{T}_-$ which has
$\phi_{1,n}$, ${\ti \phi}_{1,m}$
as one-sided orthogonal polynomials, and hence similarly for the monic polynomials $\Phi_{1,n}$ and $\ti \Phi_{1,n}$ defined by \eqref{eq:Szego_recur_intro}. This proves Part
\ref{thm:su2_part2} of Theorem
\ref{thm:su2}.

\section{Appendix: The Fej\'er kernel as linear approximation}\label{sec:appendix}

We show that the first order approximation to the diagonal evaluation
\begin{equation}\label{nonlinfejer}
    \frac 1{n+1} \left|{K_n(s,s)}-\overline{w (s) ^{-1}}{D_n(s,s)}\right|
\end{equation}
of the quantity in Theorem \ref{fejerconvergence} controls convergence of  the Fej\'er means.  The first order approximation is the affine linear part of the expansion into a multilinear series in $F=(F_n)$, which converges for most quantities in nonlinear Fourier analysis  \cite{tsai}, \cite{QSP_NLFA}
for finite $\|F\|_1$.
 Recall from \eqref{eq:NLFT_defn_intro}
that
\begin{equation}\label{recallphi}   
\phi_n(z) = z^{-n} (a_n(z)  + b_n ^*(z))\,  ,  \qquad 
\tilde{\phi}_n(z) =  
 z^{-n} (a_n (z) - b_n ^*(z))\, ,
\end{equation}
which we use to express the reproducing kernel \eqref{reprokernel} along its diagonal as
\begin{equation}\label{fejer1}
K_n(s,s)=\sum_{j=0}^n {\phi_j^*(s)}\tilde{\phi}_j(s)
=\sum\limits_{j=0}^n (a_j^* (s) + b_j(s)) (a_j (s) - b_j ^*(s))\, .
\end{equation}
Note also $D_n(s,s)=n+1$. 
Denote the linear Fourier series of the sequence $F$ by \[\widehat{F}(z): =\sum\limits_{k} F_k z^k\, ,\]
and similarly for any truncation of the sequence $F \mathbf{1}_{[0,j]}$.
Coming from 
 one sided orthogonal polynomials, we assume that $F_k=0$ unless $k>0$.
For small $\|F\|_1$, by \cite[(5.8), (5.9), (5.10), (5.12)]{QSP_NLFA} we have the approximations
\begin{equation}\label{eq:lin_approx_an}
a_n(z) = 1 + O(\|F\mathbf{1}_{[1,n]}\|_1^2)\, , \ a(z) = 1 + O(\|F\|_1^2)\, ,
\end{equation}
\begin{equation}
    \label{eq:lin_approx_bn}
b_n(z) = \widehat{F\mathbf{1}_{[1,n]}} (z) + O(\|F\mathbf{1}_{[1,n]}\|_1^3) \, ,\ 
b(z) = \widehat{F} (z) + O(\|F\|_1^3) \, .
\end{equation}
Recalling \eqref{measure with NLFT}, for small $\|F\|_1$ we have 
\begin{equation}\label{fejer2}
  \overline{w(s)^{-1}}=(a^*(s)+b(s))(a(s)-b^*(s)) \, ,
\end{equation}
because each factor on the right side is bounded above and below for small $\|F\|_1$ by \eqref{eq:lin_approx_an} and \eqref{eq:lin_approx_bn}.
We thus obtain for
\eqref{nonlinfejer} up to second or higher order terms
    \begin{equation*}
        \frac 1{n+1} \left| \left \{\sum_{j=0}^n 1+\widehat{F\mathbf{1}_{[0,j]}} (s)-
    \overline{\widehat{F\mathbf{1}_{[0,j]}} (s)} \right \}
        - (1+ \widehat{F} (s)-
        \overline{\widehat{F} (s)})(n+1)\right|
    \end{equation*}
     \begin{equation}\label{fejer10}
        =2 \left|\frac 1{n+1} \left \{ \sum_{j=0}^n \Im \widehat{F\mathbf{1}_{[0,j]}} (s) \right \} - \Im \widehat{F}(s)\right|\, .
    \end{equation}   
Extend $(F_n)_{n\geq0}$ to the negative integers by defining $F_{-k} : =\overline{F_k}$ for $k \geq 0$, and also call this new sequence $F$. We may then express 
\eqref{fejer10} as
\begin{equation}\label{fejer20}
        \left|\frac 1{n+1} \left \{ \sum_{j=0}^n \widehat{F\mathbf{1}_{[-j,j]}} (s) \right \} -  \widehat{F}(s)\right|\, .
    \end{equation}  
This is precisely the difference between the Fej\'er sum of $\widehat{F}$ of order $n$
and $\widehat{F}$ itself. This difference tends to zero
as $n\to \infty$ for general $\widehat{F}\in L^2(\T)$ at almost every point in $\T$.

\bibliographystyle{amsalpha}

\bibliography{references}

\end{document}